\documentclass[reqno,draft]{amsart}
\usepackage{amsfonts}
\usepackage{amsmath,amsthm}
\usepackage{amssymb}
\usepackage{euscript}   
\usepackage{color}

\newtheorem{theorem}{Theorem}[section]
\newtheorem{lemma}[theorem]{Lemma}
\newtheorem{proposition}[theorem]{Proposition}
\newtheorem{corollary}[theorem]{Corollary}

\theoremstyle{definition}
\newtheorem{example}[theorem]{Example}
\newtheorem{remark}[theorem]{Remark}

\numberwithin{equation}{section}

\let\<=\langle
\let\>=\rangle

\let\sse=\subseteq
\let\vphi=\varphi
\let\limply=\Longrightarrow

\def\0{\{0\}}

\def\conv{{\;\longrightarrow\;}}
\def\wconv{{\buildrel_{\scriptstyle w}\over\conv}}

\font\fiverm=cmr5
\def\sslash{\hbox{{\fiverm /}}}
\def\notwconv{{{\wconv\kern-13pt\sslash}\kern9pt}}

\def\A{{\mathcal A}}
\def\B{{\mathcal B}}
\def\H{{\mathcal H}}
\def\K{{\mathcal K}}
\def\Le{{\mathcal L}}
\def\M{{\mathcal M}}
\def\N{{\mathcal N}}
\def\P{{\mathcal P}}
\def\Q{{\mathcal Q}}
\def\R{{\mathcal R}}
\def\X{{\mathcal X}}

\def\A{{\acal}}
\def\H{{\hh}}
\def\K{{\kk}}
\def\Le{{\lcal}}
\def\M{{\mcal}}
\def\N{{\ncal}}
\newcommand*{\rcal}{\EuScript{R}}
\newcommand*{\xcal}{\EuScript{X}}

\def\R{{\rcal}}
\def\X{{\xcal}}

\def\BH{{\B(\hh)}}
\def\BK{{\B(\kk)}}
\def\BL{{\B(\lcal)}}

\def\CC{{\mathbb C\kern.5pt}}
\def\DD{{\mathbb D\kern.5pt}}
\def\QQ{{\mathbb Q\kern.5pt}}
\def\NN{{\mathbb N\kern.5pt}}
\def\TT{{\mathbb T\kern.5pt}}

\newcommand*{\acal}{\EuScript{A}}

\newcommand*{\cbb}{\mathbb{C}}

\newcommand*{\D}{\mathrm{d\hspace{.1ex}}}

\newcommand*{\E}{\mathrm{e}}

\newcommand*{\lcal}{\EuScript{L}}

\newcommand*{\Ge}{\geqslant}
\newcommand*{\hh}{\EuScript{H}}

\newcommand*{\is}[2]{\langle#1,#2\rangle}

\newcommand*{\kk}{\EuScript{K}}
\newcommand*{\Leq}{\leqslant}
\newcommand*{\Geq}{\geqslant}
\newcommand*{\mcal}{\EuScript{M}}
\newcommand*{\ncal}{\EuScript{N}}

\newcommand*{\ob}[1]{{\mathcal R}(#1)}

   \thanks{The research of the first and fourth
authors was supported by the National Science Center
(NCN) Grant OPUS No.\ DEC-2021/43/B/ST1/01651. The
research of the second author was supported by Basic
Science Research Program through the National Research
Foundation of Korea (NRF) funded by the Ministry of
Education (NRF-2021R111A1A01043569).}

\begin{document}

\title[Convergence of Power Sequences] {Convergence of
Power Sequences of operators via their stability}
\author
       [Z.J. Jablonski]
       {Zenon Jan Jab\l o\'nski}
\address{Instytut Matematyki, Uniwersytet Jagiello\'nski, Krak\'ow, Poland}
\email{Zenon.Jablonski@im.uj.edu.pl}
\author
       [I.B. Jung]
       {Il Bong Jung}
\address{Department of Mathematics, Kyungpook National University, Daegu,
        Korea}
\email{ibjung@knu.ac.kr}
\author
       [C.S. Kubrusly]
       {Carlos Kubrusly}
\address{Department of Electrical Engineering, Catholic University,
         Rio de Janeiro, Brazil}
\email{carlos@ele.puc-rio.br}
\author
       [J. Stochel]
       {Jan Stochel}
\address{Instytut Matematyki, Uniwersytet Jagiello\'nski, Krak\'ow, Poland}
\email{Jan.Stochel@im.uj.edu.pl}

\renewcommand{\keywordsname}{Keywords}


\date{}

   \begin{abstract}
This paper is concerned with the convergence of power
sequences and stability of Hilbert space operators,
where ``convergence'' and ``stability'' refer to weak,
strong and norm topologies. It is proved that an
operator has a convergent power sequence if and only
if it is a (not necessarily orthogonal) direct sum of
an identity operator and a stable operator. This
reduces the issue of convergence of the power sequence
of an operator $T$ to the study of stability of $T$.
The question of when the limit of the power sequence
is an orthogonal projection is investigated. Among
operators sharing this property are hyponormal and
contractive ones. In particular, a hyponormal or a
contractive operator with no identity part is stable
if and only if its power sequence is convergent. In
turn, a unitary operator has a weakly convergent power
sequence if and only if its singular-continuous part
is weakly stable and its singular-discrete part is the
identity. Characterizations of the convergence of
power sequences and stability of subnormal operators
are given in terms of semispectral measures.
   \end{abstract}
   \maketitle
\section{Introduction}
   The notion of operator stability is linked to a
discrete, time-invariant, free, linear dynamical
system modelled by the autonomous homogeneous
difference equation:
   \begin{align} \label{ldsix}
x_{n+1} = T x_n, \quad n=0,1,2, \ldots,
   \end{align}
with the initial condition $x_0 = x \in \X$, where
$\X$ is a normed space, $x$ is a vector in $\X$ and
$T$ is a bounded linear operator on $\X$, whose
solution is given by the formula
   \begin{align*}
x_n = T^n x, \quad n=0,1,2, \ldots.
   \end{align*}
The above discrete system is asymptotically stable if
$\{T^n x\}_{n=1}^{\infty}$ converges to zero for every
initial condition $x$. Of course, the meaning of the
term ``convergence'' depends on the topology with
which $\X$ is equipped. For a given operator $T$, the
operator-valued power sequence
$\{T^n\}_{n=1}^{\infty}$ of $T$ may converge to zero,
in the sense that the $\X$-valued sequence $\{T^n
x\}_{n=1}^{\infty}$ converges to zero for every $x$,
either in the weak or norm topology of $\X$, or in the
uniform norm topology of the corresponding operator
space, giving rise to the notions of weak, strong, and
uniform stability for the operator $T$. (These notions
will be clarified in Section~\ref{Sec.2}.) Thus, weak
stability refers to the weakest way in which the
linear discrete system \eqref{ldsix} approaches
asymptotically zero for all initial conditions. In
infinite dimensions (which is our case), this concept
has a continuous counterpart linked to mild solutions
of certain partial differential equations. We refer
the reader to \cite{Zab74} for a beautiful
presentation of the limit properties of discrete-time
distributed parameter systems and to \cite{Fuhr81} for
the general theory of discrete time invariant linear
systems.

The purpose of this paper is to study the convergence
of power sequences and stability of operators acting
on complex Hilbert spaces. Convergence and stability
refer to weak, strong and norm topologies. One of the
main results of the paper, Theorem~\ref{T.1}, states,
inter alia, the following.
   \begin{enumerate}
   \item[$0^{\circ}$]
{\em An operator has a weakly $($strongly, norm$)$
convergent power sequence if and only if it is a
$($not necessarily orthogonal\/$)$ direct sum of two
operators, the first of which is the identity on a
precisely described space, and the second is weakly
$($strongly, uniformly$)$ stable.}
   \end{enumerate}
In other words, the issue of weak $($strong, norm$)$
convergence of power sequences reduces to the study of
weak $($strong, uniform$)$ stability. As an immediate
consequence, we obtain that an operator is weakly
$($strongly, uniformly$)$ stable if and only if its
power sequence converges weakly $($strongly, in
norm$)$ and $1$ is not its eigenvalue (see
\cite[Theorem~1]{Kub89}). In fact, operators whose
power sequences are weakly $($strongly, norm$)$
convergent are completely characterized by the
similarity to operators whose power sequences converge
weakly $($strongly, in norm$)$ to orthogonal
projections (see Corollary~\ref{Synyr}). For this
reason, a significant part of the paper is devoted to
the situation when the weak $($strong, norm$)$ limit
of the power sequence of an operator is an orthogonal
projection (see Section~\ref{Sec.3.5}). Below are some
important excerpts from this article.
   \begin{enumerate}
   \item[$1^{\circ}$] {\it A hyponormal or a contractive operator with
no identity part is weakly $($strongly, uniformly$)$
stable if and only if its power sequence converges
weakly $($strongly, in norm$)$} (see
Corollary~\ref{C.1Z}).
   \item[$2^{\circ}$]
{\it If the power sequence of an operator $T$ with
$\liminf_{n\to \infty}\|T^n\|\Leq 1$ is weakly
convergent, then its limit is an orthogonal projection
which commutes with $T$} (see Theorem~\ref{L.1}).
   \item[$3^{\circ}$] {\it A unitary operator has a weakly convergent power
sequence if and only if its singular-continuous part
is weakly stable and its singular-discrete part is the
identity operator\/{\em ;} each part may be absent}
(see Theorem~\ref{T2}).
   \item[$4^{\circ}$] {\it Characterizations of the weak
$($strong, uniform$)$ convergence of power sequences
and the corresponding stability of subnormal operators
are stated in terms of their semispectral measures}
(see Section~\ref{Sec.5}).
   \end{enumerate}

The present paper was partially inspired by
\cite{CJJS2} and \cite{JJS23}. All concepts used above
will be defined in the subsequent sections$.$ The
paper is organized as follows$.$ Basic notation and
terminology are summarized in Section~\ref{Sec.2}$.$
The weak (strong, norm) convergence of power sequences
and the corresponding stability of general operators
are studied in Section~\ref{Sec.3}. In
Section~\ref{Sec.3.5}, we investigate the question of
when the weak limit of the power sequence of an
operator is an orthogonal projection. In particular,
regarding $2^{\circ}$, we give an example of a weakly
stable operator $T$ such that $\liminf_{n\to
\infty}\|T^n\| = 1$ and $\limsup_{n\to \infty}
\|T^n\|$ takes a predetermined numerical value from
the open interval $(1,\infty)$ (see
Example~\ref{serzcw}). The special case of unitary
operators is treated in Section~\ref{Sec.4}$.$ The
weak (strong, uniform) convergence of power sequences
and the related stability of subnormal operators is
investigated in Section~\ref{Sec.5}.
\section{\label{Sec.2}Notation and Terminology}
We write $\TT=\{z\in\CC\colon |z|=1\}$ and
$\DD=\{z\in\CC\colon |z|<1\}.$ By an operator $T$ on a
complex Hilbert space $\H$ with inner product
$\<\,\cdot,\mbox{-}\>$ we mean a bounded linear
transformation of $\H$ into itself$.$ Let $\BH$ stand
for the $C^*$-algebra of all operators on $\H.$ We use
the same symbol $\|\cdot\|$ for the norm on $\H$ and
for the induced operator norm on $\BH.$ Let $T \in
\B(\H)$. The kernel and the range of $T$ are denoted
by $\N(T)$ and $\R(T)$, respectively$.$ We write
$\sigma(T)$, $r(T)$ and $w(T)$ for the spectrum, the
spectral radius and the numerical radius of $T$,
respectively. The operator $T$ is a {\em normaloid} if
$r(T)=\|T\|$, a {\em contraction} (or a {\em
contractive operator\/}) if $\|T\|\Leq 1$, a {\em
strict contraction} if $\|T\|<1$, and {\em power
bounded} if $\sup_{n\Ge1}\|T^n\|<\infty.$ The power
sequence $\{T^n\}_{n=1}^{\infty}$ of $T$ converges
{\em weakly} if the $\H$-valued sequence
$\{T^nx\}_{n=1}^{\infty}$ converges weakly for every
$x\in\H$, that is, the complex valued sequence
$\{\<T^nx,y\>\}_{n=1}^{\infty}$ converges for all
$x,y\in\H.$ Using the uniform boundedness principle
and the Riesz representation theorem one can verify
that the power sequence $\{T^n\}_{n=1}^{\infty}$
converges weakly if and only if there exists an
operator $A\in\BH$ such that $\{T^nx\}_{n=1}^{\infty}$
converges weakly to $Ax$ for every $x\in\H$ (notation:
$T^nx \wconv Ax)$, or equivalently that
$\<T^nx,y\>\to\<Ax,y\>$ as $n\to \infty$ for all
$x,y\in\H.$ If this is the case, we say that
$\{T^n\}_{n=1}^{\infty}$ converges {\em weakly} to $A$
(notation$:$ $T^n\wconv A).$ The power sequence
$\{T^n\}_{n=1}^{\infty}$ converges {\em strongly} to
an operator $A\in \BH$ if $\{T^n x\}_{n=1}^{\infty}$
converges to $Ax$ in norm for every $x\in \H$. The
operator $T$ is said to be {\it weakly $($strongly,
uniformly$)$ stable} if the power sequence
$\{T^n\}_{n=1}^{\infty}$ converges weakly $($strongly,
in norm$)$ to the zero operator. The following fact is
well known (cf.\ \cite[p.\ 10]{MDOT}).
   \begin{align} \label{ustypiy}
   \begin{minipage}{74ex}
{\em An operator $T\in \B(\H)$ is uniformly stable if
and only if $r(T)<1$. Moreover, if $T$ is a normaloid,
then $T$ is uniformly stable if and only if
$\|T\|<1$.}
   \end{minipage}
   \end{align}
By the uniform boundedness principle, the weak
convergence of $\{T^n\}_{n=1}^{\infty}$ implies the
power boundedness of $T$. Hence, using the spectral
radius formula, we obtain.
   \begin{align} \label{gwpi}
   \begin{minipage}{65ex}
{\em If $T\in \B(\H)$ is normaloid and
$\{T^n\}_{n=1}^{\infty}$ is weakly convergent, then
$T$ is a contraction.}
   \end{minipage}
   \end{align}
We refer the reader to \cite{Eis10} and \cite{MDOT}
for more information on this subject.

Recall that an operator $T\in\BH$ is {\em normal} if
$TT^*=T^*T$, {\em hyponormal} if $TT^*\Leq T^*T$, an
{\em isometry} if $T^*T=I$, and {\em unitary} if
$TT^*=T^*T=I$, where $I$ stands for the identity
operator. The operator $T$ is {\em subnormal} if $T$
is the restriction of a normal operator to its
invariant subspace (i.e., a closed vector subspace).
If the only subspace of $\K$ reducing $N$ and
containing $\H$ is $\K$ itself, then $N$ is called a
{\em minimal normal extension} of $T$. These classes
of operators are related to each other as follows (see
\cite[Propositions~II.4.2 and II.4.6]{Co91}).
   \begin{align} \label{shnyr}
   \begin{minipage}{70ex}
{\em Subnormal operators are hyponormal and hyponormal
operators are normaloid.}
   \end{minipage}
   \end{align}
More details on the above-mentioned classes of
operators can be found in \cite{Co91}.

If $\X$ is a subset of $\H$, then we write
$\X^\perp=\{x\in \H \colon x \perp \X\}.$ Let $T\in
\B(\H)$. A subspace $\M$ of a Hilbert space $\H$ is
reducing for $T$ (or $\M$ reduces $T$) if $\M$ is
invariant for both $T$ and $T^*$ or, equivalently, if
$\M$ and $\M^\perp$ are both invariant for $T$. A {\em
part} of $T$ is understood as the restriction of $T$
to any of its reducing subspaces (Sometimes the term
``part'' is defined as the restriction of an operator
to its invariant subspace, but this is not our case.).
Weak convergence behaves well with respect to
orthogonal parts and sums. Indeed, if $\K\oplus\lcal$
is the orthogonal sum of Hilbert spaces $\K$ and
$\Le$, $S\oplus T$ and $A\oplus B$ are the orthogonal
sums of operators $S,A\in\BK$ and $T,B\in\BL$, then
$(S\oplus T)^n = S^n\oplus T^n$ for every integer
$n\Ge 0$~and
   \begin{align} \label{(1)}
\textit{$(S\oplus T)^n\wconv A\oplus B$ \; if and only
if \; $S^n\wconv A$ and $T^n\wconv B$.}
   \end{align}
In particular, $S\oplus T$ is weakly stable if and
only if $S$ and $T$ are weakly stable. The same
properties are shared by strong and operator norm
topologies. Similar assertions can be stated for
direct sums of operators, as discussed in
Section~\ref{Sec.3}.
\section{\label{Sec.3}Convergence of Power Sequences}
In this section, we extend a result from \cite{CJJS2}
on the weak convergence of power sequences of unitary
operators, showing that it is in fact valid for
arbitrary operators, however, at the cost of losing
the orthogonality of the limit projection (see
Theorem~\ref{T.1}). For the convenience of the reader,
we state it explicitly.
\begin{lemma}[\mbox{\cite[Lemma~3.6]{CJJS2}}] \label{P.1}
Let $U$ be a unitary operator on a Hilbert space.
   \begin{enumerate}
   \item[(i)]
If ${U^n\wconv P}$, then $P$ is an orthogonal
projection such that $\R(P)=\N(I-U)$.
   \item[(ii)]
$U$ is weakly stable if and only if ${U^n\wconv P}$
and ${\N(I-U)=\0}$.
   \end{enumerate}
   \end{lemma}
It is worth pointing out that part (ii) of
Lemma~\ref{P.1} holds for arbitrary operators (see
\cite[Theorem~1]{Kub89}).

Before we formulate the main result of this section,
we will summarize the basic facts about (not
necessarily orthogonal) projections in Hilbert spaces,
where by a {\em projection} we mean an operator $P\in
\BH$ which is an idempotent, that is, $P^2=P$.
   \begin{enumerate}
   \item[$\bullet$]  If $P\in \BH$ is a projection, then $\R(P)$
and $\R(I-P)$ are closed and $\H=\R(P) \dotplus
\R(I-P)$, where $\dotplus$ means direct (algebraic)
sum. Moreover, $\H=\R(P) \oplus \R(I-P)$ if and only
if $P$ is selfadjoint.
   \item[$\bullet$] Conversely, by the closed graph theorem, if
$\H=\H_1\dotplus \H_2$, where $\H_1$ and $\H_2$ are
subspaces of $\H$, then there exists a unique
projection $P\in \BH$ such that $\R(P)=\H_1$ and
$\R(I-P)=\H_2$.
   \item[$\bullet$] Let $\H_1$ and $\H_2$ be
subspaces of $\H$ such that $\H=\H_1\dotplus \H_2$,
$T_1\in \B(\H_1)$ and $T_2\in \B(\H_2)$. We write
$T=T_1\dotplus T_2$ for the operator $T\in \BH$ given
by $T(x_1 + x_2)=T_1 x_1 + T_2x_2$ for $x_1\in \H_1$
and $x_2\in \H_2.$ Then $TP=PT$, where $P\in \BH$ is a
unique projection with $\R(P)=\H_1$ and
$\R(I-P)=\H_2$.
   \item[$\bullet$] Conversely, if $TP=PT$ for some projection $P\in \BH$,
then $T$ decomposes as $T=T_1\dotplus T_2$, where
$T_1\in \B(\R(P))$ and $T_2\in \B(\R(I-P))$.
   \end{enumerate}
   \begin{theorem} \label{T.1}
Let $T\in \BH$. Then the following statements hold{\em
:}
   \begin{enumerate}
   \item[(i)]  if $\{T^n\}_{n=1}^{\infty}$ converges weakly
$($strongly, in norm$)$ to $P\in \BH$, then $P$ is a
projection with $\R(P)=\N(I-T)$ and $T$ decomposes as
$T=I \dotplus L$ with respect to the direct
decomposition $\H = \R(P) \dotplus \R(I-P)$, where $L$
is weakly $($strongly, uniformly$)$ stable,
   \item[(ii)] if $P\in \BH$ is a projection and
$T$ decomposes as $T=I \dotplus L$ with respect to the
direct decomposition $\H = \R(P) \dotplus \R(I-P)$,
where $L$ is weakly $($strongly, uniformly$)$ stable,
then $\{T^n\}_{n=1}^{\infty}$ converges weakly
$($strongly, in norm$)$ to $P$,
   \item[(iii)] $T$ is weakly $($strongly, uniformly$)$
stable if and only if $\N(I-T)=\0$ and
$\{T^n\}_{n=1}^{\infty}$ is weakly $($strongly,
norm$)$ convergent.
   \end{enumerate}
   \end{theorem}
   \begin{proof}
First, we deal with the case of weak topology.

(i) Suppose that $\{T^n\}_{n=1}^{\infty}$ converges
weakly to $P\in \BH$. It follows from
\cite[Theorem~1]{Kub89} that $P$ is a projection which
commutes with $T$. For the reader's convenience, we
sketch a slightly different proof of this fact. Since
$T^{n}\wconv P$ as $n\to \infty$, we see that for
$k\Geq 1$, $T^{n+k}\wconv P$ as $n\to \infty$. Using
the fact that multiplication in $\BH$ is separately
weakly continuous, we deduce that for $k\Ge 1$,
$T^{n+k}\wconv PT^k$ as $n\to \infty$. Thus, $P=PT^k$
for $k\Ge 1$. Passing to the limit as $k\to \infty$,
we conclude that $P=P^2$. Arguing as above, we get
   \begin{align} \label{komits}
PT = \textrm{(weak)} \lim_{n\to \infty} T^n T= P =
\textrm{(weak)} \lim_{n\to \infty} T T^n = TP.
   \end{align}
Hence, $P$ commutes with $T$.

Now, we show that $\R(P) = \N(I-T)$. Take $x\in\R(P)$.
Since $P=P^2$, we see that $Px=x$. Hence, by
\eqref{komits}, $(I-T)x= (I-T)Px =0$, so
$x\in\N(I-T)$. Conversely, if $x\in \N(I-T)$, then
$T^nx=x$ for all $n\Geq 1$. Passing to the limit as
$n\to \infty$ yields $Px=x$, so $x\in \R(P)$. This
shows that $\ob{P}=\N(I-T)$. By \eqref{komits}, $T$
decomposes as $T=I \dotplus L$ with respect to the
direct decomposition $\H = \R(P) \dotplus \R(I-P)$.


It remains to show that $L$ is weakly stable. For,
since $\{T^n\}_{n=1}^{\infty}$ is weakly convergent,
so is $\{L^n\}_{n=1}^{\infty}$. Applying what was
proved above to $L$ instead of $T$, we see that the
weak limit of $\{L^n\}_{n=1}^{\infty}$ is a projection
acting on $\R(I-P)$ with range $\N(I-L)$. Noting that
$\N(I-L)=\0$, we conclude that $L$ is weakly stable.

(ii) It is easy to verify that under the assumptions
of (ii), $\{T^n\}_{n=1}^{\infty}$ converges weakly to
$Q\in \BH$, where $Q$ decomposes as $Q=I \dotplus 0$
with respect to $\H = \R(P) \dotplus \R(I-P)$. This
implies that $Q=P$.

(iii) This statement follows easily from (i).

Our next goal is to cover the cases of strong and norm
topologies. Due to the similarity of the proof, we
will only consider the case of strong topology. Thus,
if $\{T^n\}_{n=1}^{\infty}$ is strongly convergent to
$P\in \BH$, then it is weakly convergent to $P$ and
so, by (i), $P$ is a projection with $\R(P)=\N(I-T)$
and $T=I \dotplus L$ with respect to $\H = \R(P)
\dotplus \R(I-P)$, where $L$ is weakly stable.
However, the strong convergence of
$\{T^n\}_{n=1}^{\infty}$ implies the strong
convergence of $\{L^n\}_{n=1}^{\infty}$, which
ultimately, due to the uniqueness of limit, gives the
strong stability of $L$. This proves (i) for strong
topology. Similar arguments can be used to prove
statements (ii) and (iii) in the case of strong
topology. This completes the proof.
   \end{proof}
   \begin{remark}
It is a matter of routine to verify that statements
(i) and (ii) of Theorem~\ref{T.1} are equivalent to
(i$^{\prime}$) and (ii$^{\prime}$), respectively,
where
   \begin{enumerate}
   \item[(i$^\prime$)] if $\{T^n\}_{n=1}^{\infty}$
converges weakly $($strongly, in norm$)$ to $P\in
\BH$, then $P$ is a projection with $\R(P)=\N(I-T)$
such that $TP=PT$ and $T|_{\R(I-P)}$ is weakly
$($strongly, uniformly$)$ stable,
   \item[(ii$^{\prime}$)] if $P\in \BH$ is a projection with
$\R(P)=\N(I-T)$ such that $TP=PT$ and $T|_{\R(I-P)}$
is weakly $($strongly, uniformly$)$ stable, then
$\{T^n\}_{n=1}^{\infty}$ converges weakly $($strongly,
in norm$)$ to $P$.
   \hfill $\diamondsuit$
   \end{enumerate}
   \end{remark}
   \begin{corollary} \label{Synyr}
Let $T\in \BH$. Then $\{T^n\}_{n=1}^{\infty}$
converges weakly $($strongly, in norm$)$ to $P\in\BH$
if and only if $T$ is similar to an operator $R\in
\BK$ with the property that $\{R^n\}_{n=1}^{\infty}$
converges weakly $($strongly, in norm$)$ to an
orthogonal projection $Q\in \BK.$ If this is the case,
then $P$ and $Q$ are similar via the same similarity
as $T$ and $R$.
   \end{corollary}
   \begin{proof}
Suppose that $\{T^n\}_{n=1}^{\infty}$ converges weakly
$($strongly, in norm$)$ to $P\in\BH$. Then by
Theorem~\ref{T.1}(i), $P$ is a projection. Set
$\K=\R(P) \oplus \R(I-P)$ (the exterior orthogonal
sum) and define the operator $S\in \B(\K,\H)$ by $S(x,
y)= x+y$ for $x\in \R(P)$ and $y\in \R(I-P).$ Since
$P$ is a projection, we have $\H=\R(P) \dotplus
\R(I-P)$, so by the inverse mapping theorem, $S$ is
invertible. It is easy to verify that the operator
$Q:=S^{-1}PS$ is an orthogonal projection. Set
$R=S^{-1}TS$. Since the~map
   \begin{align*}
\BH\ni X \longmapsto S^{-1}XS\in \BK
   \end{align*}
is a unital algebra isomorphism which is a weak
$($strong, norm$)$ homeomorphism, we see that
$\{R^n\}_{n=1}^{\infty}$ converges weakly $($strongly,
in norm$)$ to $Q\in\BH$. Reversing the above reasoning
completes the proof.
   \end{proof}
   \begin{remark}
Regarding the recently published
\cite[Theorem~2.1]{CJJS2}, it is worth noting that the
limit operators $P$ and $A$ appearing there are,
respectively, the orthogonal projection of
$\hh_{1\mathrm{u}}$ onto $\N(I-U)$ and the idempotent
with $\R(A)=\N(I-X)$. The first fact is a direct
consequence of Lemma~\ref{P.1}. The second follows
immediately from Theorem~\ref{T.1}.
   \hfill $\diamondsuit$
   \end{remark}
\section{\label{Sec.3.5}Orthogonality of limit projection}
In this section, we generalize part (i) of
Lemma~\ref{P.1} to cover the cases of hyponormal and
contractive operators (see Corollaries~\ref{owcp} and
\ref{C.1Z}, see also Remark~\ref{dymint}), as well as
operators $T$ with $\liminf_{n\to \infty}\|T^n\|\Leq
1$ (see Theorem~\ref{L.1}; see also
Remark~\ref{mynied}).

In view of Theorem~\ref{T.1}(i), if the power sequence
$\{T^n\}_{n=1}^{\infty}$ of an operator $T$ converges
weakly to $P$, then $P$ is a projection. On the other
hand, by Corollary~\ref{Synyr}, such a $T$ is similar
to an operator whose power sequence is weakly
convergent to an orthogonal projection. It is
therefore of interest to answer the question of when
the weak limit of power sequence is an orthogonal
projection. We will begin with the following general
observation that will shed more light on the answer to
the above question.
   \begin{proposition} \label{wqaw}
Let $T\in \B(\H)$. Then the following conditions are
equivalent{\em :}
   \begin{enumerate}
   \item[(i)] $\N(I-T)$ reduces $T$,
   \item[(ii)] $T^*|_{\N(I-T)}$ is the identity
operator,
   \item[(iii)] $T^*|_{\N(I-T)}$ is an isometry,
   \item[(iv)] $\N(I-T)\subseteq \N(I-T^*)$.
   \end{enumerate}
   \end{proposition}
   \begin{proof}
First, observe that
   \begin{align} \label{igsq}
I|_{\N(I-T)}=(T|_{\N(I-T)})^* = Q T^*|_{\N(I-T)},
   \end{align}
where $Q\in \B(\H)$ is the orthogonal projection of
$\H$ onto $\N(I-T)$.

(i)$\Rightarrow$(ii) This is a direct consequence of
\eqref{igsq}.

(ii)$\Rightarrow$(iii) Trivial.

(iii)$\Rightarrow$(i) Note that
   \begin{align*}
\|x\| \overset{\eqref{igsq}}= \|QT^*x\| \Leq \|T^*x\|
\overset{\text{(iii)}}= \|x\|, \quad x\in \N(I-T).
   \end{align*}
Hence $\|QT^*x\| = \|T^*x\|$ for every $x\in \N(I-T)$.
This implies that $\N(I-T)$ is invariant for $T^*$,
and thus $\N(I-T)$ reduces $T$.

The equivalence (ii)$\Leftrightarrow$(iv) is obvious.
   \end{proof}
The general answer to our question is as follows.
   \begin{theorem} \label{pytxly}
Suppose that $T, P\in \B(\H)$ and $T^n\wconv P$ as
$n\to\infty$. Then the following statements are
equivalent{\em :}
   \begin{enumerate}
   \item[(i)]
$P$ is an orthogonal projection,
   \item[(ii)]
$\N(I-T)=\N(I-T^*)$,
   \item[(iii)] $\N(I-T)$ reduces $T$.
   \end{enumerate}
   If {\em (i)} holds, then $P$ is the
orthogonal projection of $\H$ onto $\N(I-T)$.
   \end{theorem}
   \begin{proof}
(i)$\Rightarrow$(ii) Suppose that $P$ is an orthogonal
projection. Using the fact that the adjoint operation
on $\B(\H)$ is weakly continuous, we see that
$T^{*n}\wconv P^*$ as $n\to\infty$. Since $P=P^*$, we
deduce from Theorem~\ref{T.1}(i) applied to $T$ and
$T^{*}$ that
   \begin{align*}
\N(I-T)=\R(P)=\R(P^*)=\N(I-T^*).
   \end{align*}

(ii)$\Rightarrow$(iii) This implication is obvious.

(iii)$\Rightarrow$(i) By (iii), $T$ decomposes as
   \begin{align} \label{iltsX}
T=I\oplus L
   \end{align}
with respect to the orthogonal decomposition $\H =
\N(I-T) \oplus \N(I-T)^\perp.$ Since
$\{T^n\}_{n=1}^{\infty}$ is weakly convergent, we
infer from \eqref{(1)} and \eqref{iltsX} that
$\{L^n\}_{n=1}^{\infty}$ is also weakly convergent.
Noting that $\N(I-L) = \{0\}$, we deduce from
Theorem~\ref{T.1}(iii) that $L$ is weakly stable.
Using \eqref{iltsX} again, we conclude that
$\{T^n\}_{n=1}^{\infty}$ converges weakly to $I\oplus
0$, which is the orthogonal projection of $\H$ onto
$\N(I-T)$.
%
   \end{proof}
It is worth mentioning that implication
(ii)$\Rightarrow$(iii) of Theorem~\ref{pytxly} is
valid for any Hilbert space operator $T$ regardless of
whether the sequence $\{T^n\}_{n=1}^{\infty}$ is
weakly convergent or not (see also
Proposition~\ref{wqaw}). The proofs of the remaining
implications (i)$\Rightarrow$(ii) and
(iii)$\Rightarrow$(i) appeal to Theorem~\ref{T.1}.

With $\varsubsetneq$ denoting proper inclusion, we
obtain the following result, which is a direct
consequence of Proposition~\ref{wqaw} and
Theorem~\ref{pytxly}.
   \begin{corollary} \label{geqpm}
If $T\in \B(\H)$ is such that $\N(I-T)\varsubsetneq
\N(I-T^*)$, then $\N(I-T)$ reduces $T$ and the power
sequence $\{T^n\}_{n=1}^{\infty}$ is not weakly
convergent to an orthogonal projection, and thus $T$
is not weakly stable.
   \end{corollary}
An explicit example of an operator $T\in \B(\H)$ for
which $\N(I-T)$ reduces $T$, but $\N(I-T)\varsubsetneq
\N(I-T^*)$, is given in Example~\ref{pvzxk}.

According to Theorem~\ref{pytxly}, if the power
sequence $\{T^n\}_{n=1}^{\infty}$ of an operator $T$
converges weakly to $P$, then the requirement that
$\N(I-T)$ reduces $T$ is necessary and sufficient for
$P$ to be an orthogonal projection. Under this
assumption, the weak convergence of the power sequence
$\{T^n\}_{n=1}^{\infty}$ is completely determined by
the weak stability of $L$, where $L$ is as in
\eqref{iltsX}.
   \begin{corollary}\label{owcp}
Let $T\in\BH$ be an operator such that $\N(I-T)$
reduces $T$ and let
   \begin{align*}
T=I\oplus L
   \end{align*}
be the orthogonal decomposition of $T$ with respect to
the decomposition
   \begin{align*}
\H = \N(I-T) \oplus \N(I-T)^\perp.
   \end{align*}
Then the power sequence $\{T^n\}_{n=1}^{\infty}$ is
weakly $($strongly, norm\/$)$ convergent if and only
if $L$ is weakly $($strongly, uniformly\/$)$ stable,
and if this is the case, then the weak $($strong,
norm\/$)$ limit of $\{T^n\}_{n=1}^{\infty}$ is the
orthogonal projection of $\H$ onto $\N({I-T})$ and
$\N(I-T)= \N(I-T^*)$.
   \end{corollary}
   \begin{proof}
As in the proof of implication (iii)$\Rightarrow$(i)
of Theorem~\ref{pytxly}, we see that if the power
sequence $\{T^n\}_{n=1}^{\infty}$ is weakly (strongly,
norm) convergent, then $L$ is weakly (strongly,
uniformly) stable. The converse implication is
obvious. It follows from Theorem~\ref{pytxly} that the
weak (strong, norm) limit of $\{T^n\}_{n=1}^{\infty}$
is the orthogonal projection of $\H$ onto $\N({I-T})$.
   \end{proof}
   \begin{corollary} \label{C.1Z}
Let $T\in\BH$ be a hyponormal $($resp., contractive$)$
operator. Then
   \begin{enumerate}
   \item[(i)] $\N(I-T)$ reduces $T$ and $T$ decomposes as
$T=I\oplus L$ with respect to the orthogonal
decomposition $\H = \N(I-T) \oplus \N(I-T)^\perp$,
where $L$ is a hyponormal $($resp., contractive$)$
operator on $\N(I-T)^\perp,$
   \item[(ii)] $\{T^n\}_{n=1}^{\infty}$ is weakly
$($strongly, norm\/$)$ convergent if and only if $L$
is weakly $($strongly, uniformly\/$)$ stable, and if
this is the case, then the weak $($strong, norm\/$)$
limit of $\{T^n\}_{n=1}^{\infty}$ is the orthogonal
projection of $\H$ onto $\N({I-T})$ and $\N(I-T)=
\N(I-T^*)$.
   \end{enumerate}
   \end{corollary}
   \begin{proof}
If $T$ is hyponormal, then the restriction of $T$ to
its invariant subspace is hyponormal and $\N(I-T) \sse
\N(I-T^*)$ (see \cite[Proposition~II.4.4]{Co91}), so
$\N(I-T)$ reduces $T$. In turn, if $T$ is contractive,
then by \cite[Proposition~I.3.1]{Sz.-N-Fo70},
$\N(I-T)=\N(I-T^*)$, so $\N(I-T)$ reduces $T$. Now we
can apply Corollary~\ref{owcp}.
   \end{proof}
For an operator $T\in \B(\H)$ reduced by $\N(I-T)$,
the restriction of $T$ to $\N(I-T)$ is called the {\em
identity part} of $T$. By Corollary~\ref{C.1Z}, we get
the following.
   \begin{align*}
   \begin{minipage}{70ex}
{\em A hyponormal or a contractive operator with no
identity part is weakly $($strongly, uniformly$)$
stable if and only if its power sequence converges
weakly $($strongly, in norm$)$.}
   \end{minipage}
   \end{align*}
   \begin{remark} \label{dymint}
There are more classes of Hilbert space operators $T$
for which the eigenspace $\N(I-T)$ reduces $T$. For
instance, this is the case for operators having the
property that $\N(\alpha I-T)\sse \N(\overline\alpha
I-T^*)$ for every $\alpha\in\CC$. This property, in
turn, is possessed by so-called {\em dominant
operators}, i.e., operators $T$ such that ${\R(\alpha
I-T)} \sse \R(\overline\alpha I-T^*)$ for every
$\alpha \in \CC$ (see \cite{Stam77}). Let us mention
that the class of dominant operators on an infinite
dimensional Hilbert space is essentially larger than
the class of hyponormal operators. For the discussion
of the case of numerical contractions, i.e., operators
$T$ with $w(T)\Leq 1$ (less restrictive than
$\|T\|\Leq 1$), we refer the reader to
Remark~\ref{mynied}.
   \hfill $\diamondsuit$
   \end{remark}
   \begin{example} \label{pvzxk}
Let $\H$ be an infinite-dimensional separable Hilbert
space and let $\{e_n\}_{n=0}^{\infty}$ be an
orthonormal basis of $\H$. Take a bounded sequence
$\{\lambda_n\}_{n=0}^{\infty}$ of positive real
numbers. Then there exists a unique operator $T\in
\BH$, called a weighted shift (with weights
$\{\lambda_n\}_{n=0}^{\infty}$), such that
$Te_n=\lambda_n e_{n+1}$ for all integers $n \Geq 0$.
Assume that the sequence
$\{\lambda_n\}_{n=0}^{\infty}$ is monotonically
increasing to $\lambda_{\infty}\in (1,\infty)$. Then
the weighted shift $T$ is hyponormal (see \cite[p.\
83, Lemma]{shi74}). It is well known that the point
spectrum of $T$ is empty (see
\cite[Theorem~8(i)]{shi74}), so $\N(I-T)=\{0\}$
reduces $T$. It follows from \cite[Theorem~1]{Sta66}
that the point spectrum of $T^*$ is equal to $\{z\in
\cbb\colon |z| < \lambda_{\infty}\}$. Since
$\lambda_{\infty} > 1$, we see that $1$ is in the
point spectrum of $T^*$. Hence, $\N(I-T) \varsubsetneq
\N(I-T^*)$, thus by Corollary~\ref{geqpm} the power
sequence $\{T^n\}_{n=1}^{\infty}$ is not weakly
convergent to an orthogonal projection and so $T$ is
not weakly stable. Using \cite[Theorem~5]{Sta66} and
the fact that subnormal operators are hyponormal, we
can modify the weights $\{\lambda_n\}_{n=0}^{\infty}$
of $T$ so that $T$ is not only hyponormal, but also
subnormal (see also \cite[Proposition~25]{shi74}).
Finally, by considering the operator $I\oplus T$ with
the above $T$, we obtain an operator $S$ with the
property that $\N(I-S)$ reduces $S$, $\N(I-S)\neq
\{0\}$ and $\N(I-S) \varsubsetneq \N(I-S^*)$.
   \hfill $\diamondsuit$
   \end{example}
In Theorem~\ref{L.1} below, we will continue the
discussion of the question of when the weak limit of
the power sequence of an operator $T$ is an orthogonal
projection. We will give a sufficient condition for
this, assuming that $\liminf_{n\to\infty} \|T^n\| \Leq
1$ (see also Remark~\ref{mynied}). Before stating the
result, we show that the new assumption is closely
related to the spectral radius of $T$ and to the
uniform stability of $T$. For the reader's
convenience, we sketch the proof of the equivalence of
conditions (i)-(iv) below (cf.\ the proof of
\cite[Theorem~3]{Ku-Vi08}).
   \begin{lemma} \label{limyng}
Suppose that $T \in \B(\H)$. Then the following
conditions are equivalent{\em :}
   \begin{enumerate}
   \item[(i)] $T$ is uniformly stable,
   \item[(ii)] $\liminf_{n\to \infty}\|T^n\| < 1$,
   \item[(iii)] $\inf_{n\Geq k}\|T^n\| < 1$ for some integer $k\Ge 1$,
   \item[(iv)] $r(T)<1$.
   \end{enumerate}
Moreover, if $\liminf_{n\to \infty}\|T^n\| = 1$, then
$r(T)=1$.
   \end{lemma}
   \begin{proof} The implications (i)$\Rightarrow$(ii)
and (ii)$\Rightarrow$(iii) are trivial.

(iii)$\Rightarrow$(iv) Since $\inf_{n\Geq k}\|T^n\| <
1$, there exists an integer $\ell\Ge k$ such that
$\|T^{\ell}\| < 1$. Then, by the spectral radius
formula (see \cite[Theorem~10.13]{Rud}), we get
   \begin{align*}
r(T)= \inf_{n \Geq 1} \|T^n\|^{1/n} \Leq
\|T^{\ell}\|^{1/\ell} < 1.
   \end{align*}

(i)$\Leftrightarrow$(iv) This equivalence is well
known (see \cite[p.\ 10]{MDOT}).

Assume that $\liminf_{n\to \infty}\|T^n\| =1$. Then
there exists a subsequence
$\{\|T^{k_n}\|\}_{n=1}^{\infty}$ tending to $1$. Since
$\lim_{n\to\infty} \alpha_n=1$ implies that
$\lim_{n\to\infty} \alpha_n^{1/\ell_n}=1$ for any
sequence $\{\alpha_n\}_{n=1}^{\infty}$ of positive
real numbers and any sequence
$\{\ell_n\}_{n=1}^{\infty}$ of positive integers
tending to $\infty$, we deduce that $\lim_{n\to\infty}
\|T^{k_n}\|^{1/\ell_n} = 1$. Applying this to
$\ell_n=k_n$ and using the spectral radius formula, we
see that
   \begin{align*}
r(T) = \lim_{n\to\infty} \|T^n\|^{1/n} =
\lim_{n\to\infty} \|T^{k_n}\|^{1/k_n} = 1.
   \end{align*}
This completes the proof.
   \end{proof}
   \begin{corollary}
If $T \in \B(\H)$, then the following conditions are
equivalent{\em :}
   \begin{enumerate}
   \item[(i)] $\liminf_{n\to \infty}\|T^n\| \Leq 1$,
   \item[(ii)] either $T$ is uniformly stable or
$\liminf_{n\to \infty}\|T^n\|=1$.
   \end{enumerate}
   \end{corollary}
As shown in an example below, the converse of the
implication
   \begin{align*}
\liminf_{n\to \infty}\|T^n\| = 1 \implies r(T)=1
   \end{align*}
appearing in the ``moreover'' part of
Lemma~\ref{limyng} does not hold in general.
   \begin{example}
Let $T\in \B(\H)$ be a {\em $2$-isometry} on a Hilbert
space $\H$, that is, $I-2T^*T + T^{*2}T^{2}=0$. It
follows from \cite[Proposition~4.5]{Jab02} that there
exists a positive operator $C\in \B(\H)$ such that
   \begin{align} \label{iplnc}
\text{$T^{*n}T^n = I + n C$ for all integers $n\Ge
0$.}
   \end{align}
Using the spectral radius formula, we deduce that
$r(T)=1$ (in fact, this is a consequence of a more
general fact, see \cite[Lemma~1.21]{Ag-St95-6}).
Assume that $T$ is not an isometry, that is, $C \neq
0$. Then, by \eqref{iplnc}, $\lim_{n\to\infty}
\|T^n\|=\infty$. To have an example of a non-isometric
$2$-isometry, consider the weighted shift $T$ with
weights $\{\sigma_n(\lambda)\}_{n=0}^{\infty}$ defined
by
   \begin{align*}
\sigma_n(\lambda) = \sqrt{\frac{1 +
(n+1)(\lambda^2-1)} {1 + n(\lambda^2-1)}}, \quad n \Ge
0,
   \end{align*}
where $\lambda \in (1,\infty)$ (see, e.g.,
\cite[Lemma~6.1]{Jab-St01}).
   \hfill $\diamondsuit$
   \end{example}
Next, we need the following characterization of
orthogonal projections.
   \begin{lemma}[{\cite[Lemma]{Fur-Nak71}}] \label{prichyr}
Let $\H$ be a Hilbert space and $P\in \B(\H)$ be an
idempotent. Then the following conditions are
equivalent{\em :}
   \begin{enumerate}
   \item[(i)] $P$ is an orthogonal projection,
   \item[(ii)] $|\is{Px}{x}| \Leq \|x\|^2$ for all $x\in
\H$, or, equivalently, $w(P) \Leq 1$.
   \end{enumerate}
   \end{lemma}
   Now we provide yet another criterion for the limit
of a weakly convergent power sequence of an operator
to be an orthogonal projection.
   \begin{theorem} \label{L.1}
Suppose that $T, P\in \B(\H)$ are such that $T^n\wconv
P$ as $n\to\infty$ and
   \begin{align} \label{lbnunf}
\liminf_{n\to \infty}\|T^n\|\Leq 1.
   \end{align}
Then $P$ is the orthogonal projection of $\H$ onto
$\N(I-T)$, $\N(I-T)$ reduces $T$ and
$\N(I-T)=\N(I-T^*)$.
   \end{theorem}
   \begin{proof}
By Theorem~\ref{T.1}(i), $P=P^2$ and $\R(P)=\N(I-T)$.
Since
   \begin{align*}
|\<Px,x\>|=\lim_{n\to\infty} |\<T^nx,x\>| \Leq
\liminf_{n\to \infty}\|T^n\| \|x\|^2 \Leq \|x\|^2,
\quad x\in \H,
   \end{align*}
we deduce from Lemma~\ref{prichyr} that $P$ is the
orthogonal projection of $\H$ onto $\N(I-T)$. The
remaining part of the conclusion is a direct
consequence of Theorem~\ref{pytxly}.
   \end{proof}
   \begin{remark} \label{mynied}
Arguing as in the above proof, we see that
Theorem~\ref{L.1} remains valid if the assumption
\eqref{lbnunf} is replaced by
   \begin{align} \label{lynf}
\liminf_{n\to \infty}\|T^n x\|\Leq \|x\|, \quad x\in
\H.
   \end{align}
Repeating the above reasoning again, we deduce that
Theorem~\ref{L.1} is also true if the assumption
\eqref{lbnunf} is replaced by
   \begin{align} \label{nynrsag}
\liminf_{n\to \infty} w(T^n)\Leq 1.
   \end{align}
In turn, using the Berger inequality $w(T^n) \Leq
w(T)^n$, which holds for all nonnegative integers $n$
(see \cite[Problem~221]{Hal82}), we conclude that the
inequality
   \begin{align} \label{byrgura}
w(T)\Leq 1
   \end{align}
implies \eqref{nynrsag}, which gives yet another
version of Theorem~\ref{L.1}. It is obvious that
\eqref{lbnunf} implies both \eqref{lynf} and
\eqref{nynrsag}. Note also that any contraction $T$
satisfies conditions \eqref{lbnunf}-\eqref{byrgura},
and that for any $T$ with $w(T)\Leq 1$, $\N(I-T)$
reduces $T$ and $\N(I-T)=\N(I-T^*)$ (the second fact
is well known, although not easy to find in the
literature). Therefore, it is an open question which
of the conditions \eqref{lbnunf}-\eqref{nynrsag}
implies that the space $\N(I-T)$ reduces $T$ or/and
$\N(I-T)=\N(I-T^*)$.
   \hfill $\diamondsuit$
   \end{remark}
We conclude this section by showing that for every
$\vartheta \in (1,\infty]$, there exists an operator
$T\in \BH$ such that
   \begin{align} \label{witsg1}
&\text{$\liminf_{n\to \infty}\|T^n\| = 1$ (so $T$ is
not uniformly stable), }
   \\ \label{witsg2}
&\text{$\limsup_{n\to \infty} \|T^n\| = \vartheta$ (so
$T$ is not a contraction), }
   \\ \label{witsg3}
&\text{$T$ is weakly stable if and only if $\vartheta
< \infty$.}
   \end{align}
Certainly, by \eqref{witsg1} and the moreover part of
Lemma~\ref{limyng}, $r(T)=1$ regardless of whether
$\vartheta$ is finite or not.
   \begin{example} \label{serzcw}
Let $T\in \B(\H)$ be a weighted shift with positive
real weights $\{\lambda_n\}_{n=0}^{\infty}$ relative
to the orthonormal basis $\{e_n\}_{n=0}^{\infty}$ of
$\H$ (see Example~\ref{pvzxk}). Then
   \begin{gather} \label{tken}
T^k e_n = \lambda_n \cdots \lambda_{n+k-1} e_{n+k},
\quad k\Geq 1, \, n\Geq 0,
   \\ \label{wxka}
\|T^k\| =\sup_{n\Geq 0} \lambda_n \cdots
\lambda_{n+k-1}, \quad k\Geq 1.
   \end{gather}
Let $\{q_n\}_{n=1}^{\infty}$ be a strictly decreasing
sequence of positive real numbers such that
$\lim_{n\to \infty} q_n=1$. To determine the weights
that meet our requirements, we will use the inflation
method, which amounts to constructing the sequences
$\{x_j\}_{j=1}^{\infty}$, $\{t_j\}_{j=1}^{\infty}$ and
$\{l_j\}_{j=1}^{\infty}$ of integers greater than $2$
such that
   \begin{align} \label{kxpyq}
\textit{$t_j=s_jx_j$, $l_j > m_j:=t_j +
\sum_{i=1}^{j-1}(t_i+l_i)$ and $x_{j+1}>m_j+l_j$ for
all $j\Ge 1$,}
   \end{align}
where $\{s_j\}_{j=1}^{\infty}$ is an arbitrary
sequence of integers greater than $2$ (with convention
$\sum_{i=1}^{0}\xi_i=0$). Namely, using induction, we
can construct a sequence of finite segments $\P_j$ and
$\Q_j$, $j\Ge 1$, of the form
   \begin{align} \label{pugj}
\overset{\textsf{segment
$\P_j$}}{\overbrace{\underset{t_j=s_jx_j}
{\underbrace{\underset{x_j}{\underbrace{q_j, 1,
\ldots,1}}, \underset{x_j}{\underbrace{q_j, 1,
\ldots,1}},\ldots, \underset{x_j}{\underbrace{q_j, 1,
\ldots,1}},\underset{x_j}{\underbrace{q_j, 1, \ldots,
1, q_j^{-s_j}}}}}}}, \overset{\textsf{segment
$\Q_j$}}{\overbrace{\underset{l_j}{\underbrace{1,
\ldots\ldots\ldots\ldots, 1}}}},
   \end{align}
arranged in the order $\P_1, \Q_1, \P_2, \Q_2, \ldots,
\P_n, \Q_n, \ldots$, where the length of the $j$-th
segment $\P_j$ is $t_j$ and the length the $j$-th
segment $\Q_j$ is $l_j$; the $j$-th segment $\P_j$
itself is partitioned into $s_j-1$ segments of the
form $q_j, 1, \ldots,1$, each of length $x_j$, plus a
single segment of the form $q_j, 1, \ldots, 1,
q_j^{-s_j}$ of the same length $x_j$; finally, the
segment $\Q_j$ consists of $l_j$ units $1$. Having
done this, we define the weights
$\{\lambda_n\}_{n=0}^{\infty}$ of the weighted shift
$T$ as follows
   \begin{gather} \label{pyqu}
\lambda_0, \lambda_1, \lambda_2, \ldots = \P_1, \Q_1,
\P_2, \Q_2, \ldots, \P_n, \Q_n, \ldots.
   \end{gather}
By \eqref{kxpyq}, we have
   \begin{align} \label{xj1h}
x_{j+1} > m_j \Geq t_j = s_j x_j > x_j, \quad j\Geq 1.
   \end{align}
The number of occurrences of $q_j$ in the $j$-th
segment $\P_j$ is equal to $s_j$ (the expression
$q_j^{-s_j}$ is not counted). First, we prove that
$\liminf_{n\to \infty} \|T^n\| = 1$. Since
$\{q_n\}_{n=1}^{\infty}$ is decreasing, we deduce from
\eqref{wxka}-\eqref{xj1h} that
   \begin{align} \label{wnjkl}
\|T^n\| & =
   \begin{cases}
q_{j+1} & \text{ if } m_j \Leq n \Leq m_j + l_j \text{
with } j\Geq 1,
   \\[1ex]
q_j^{s_j} & \text{ if } m_j - x_j + 1 \Leq n < m_j
\text{ with } j\Geq 1,
   \end{cases}
   \end{align}
and
   \begin{align} \label{wnjkl9}
q_j \Leq \|T^n\| & \Leq q_j^{s_j} \quad \text{ if }
m_j -t_j + 1 \Leq n < m_j - x_j + 1 \text{ with } j
\Geq 1.
   \end{align}
By \eqref{wnjkl}, \eqref{wnjkl9} and the fact that
$q_j \searrow 1$ as $j\to \infty$, we have
   \begin{align} \label{lane}
\liminf_{n\to \infty} \|T^n\| = \lim_{j\to \infty} q_j
=1 \text{ and } \limsup_{n\to \infty} \|T^n\| =
\limsup_{j\to \infty} q_j^{s_j}.
   \end{align}
In particular, this proves \eqref{witsg1}. Write $q_j$
as $q_j = \E^{\varepsilon_j}$ with $\varepsilon_j=\log
q_j$. Then $\{\varepsilon_n\}_{n=1}^{\infty}$ is a
strictly decreasing sequence of positive real numbers
such that $\lim_{n\to \infty} \varepsilon_n =0$.
Assume that $\{s_j\}_{j=1}^{\infty}$ is strictly
increasing. Then $\lim_{j\to\infty} s_j=\infty$. We
show that for every $\vartheta \in (1,\infty]$, there
exists $\{\varepsilon_j\}_{j=1}^{\infty}$ such that
$\lim_{j\to \infty} s_j \varepsilon_j=\log \vartheta$.
For, consider two cases. If $\vartheta=\infty$, then
we set $\varepsilon_j = \frac{1}{\sqrt{s_j}}$ for
$j\Ge 1$. In turn, if $\vartheta\in (1,\infty)$, then
we set $\varepsilon_j = \frac{\log \vartheta}{s_j}$
for $j\Ge 1$. Since $q_j^{s_j} = \E^{s_j
\varepsilon_j}$, we deduce from \eqref{lane} that
\eqref{witsg2} holds.

As weak stability always implies uniform boundedness,
which by \eqref{witsg2} is equivalent to $\vartheta <
\infty$, it remains to prove that $T$ is weakly stable
provided $\vartheta < \infty$. Using \eqref{tken}, we
see that $\lim_{n\to \infty} \<T^n e_k,e_l\> = 0$ for
all integers $k,l \Geq 0$. Hence, $\lim_{n\to \infty}
\<T^n x,y\> = 0$ for all $x,y\in \X$, where $\X$
stands for the linear span of
$\{e_n\}_{n=0}^{\infty}$. Assume that $\vartheta <
\infty$. Since $\X$ is dense in $\H$ and, by
\eqref{witsg2}, $\sup_{n\Geq 1} \|T^n\| < \infty$, we
conclude that $T$ is weakly stable (see e.g.,
\cite[Lemma~1]{Kub89}), that is \eqref{witsg3} holds.
   \hfill $\diamondsuit$
   \end{example}
Theorem~\ref{L.1} and Corollary~\ref{C.1Z} naturally
hold for unitary operators. Outgrowths of such a
particularisation to unitary operators are the subject
of the next section.
\section{\label{Sec.4}Weak Convergence of Power Sequences of Unitary Operators}
Take a unitary operator $U\in\BH$ on a Hilbert space
$\H.$ Since $U$ is hyponormal, $\N(I-U)$ reduces $U$,
so $U$ decomposes relative to
$\H=\N(I-U)\oplus\N(I-U)^\perp$ as $U=I\oplus W$,
where $W$ is unitary on $\N(I-U)^\perp$ (any part of
the decomposition of $U$ may be absent). Thus, by
Corollary~\ref{C.1Z} we obtain the following.
   \begin{proposition}[\mbox{\cite[Corollary~3.7]{CJJS2}}]
\label{C3.7} Let $U\in\BH$ be a unitary operator on a
Hilbert space $\H$. Then $\N(I-U)$ reduces $U$ and $U$
decomposes as the orthogonal~sum
$$
U=I\oplus W
$$
relative to the orthogonal decomposition $\H = \N(I-U)
\oplus \N(I-U)^\perp$, where $W$ is a unitary operator
on $\N(I-U)^\perp.$ Moreover, $\{U^n\}_{n=1}^{\infty}$
is weakly convergent if and only if $W$ is weakly
stable, and if this is the case, then the weak limit
of $\{U^n\}_{n=1}^{\infty}$ is the orthogonal
projection of $\H$ onto $\N({I-U})$.
   \end{proposition}
Note that the power sequences $\{U^n\}_{n=1}^{\infty}$
and $\{W^n\}_{n=1}^{\infty}$ may not coverage at all;
for example, this is the case if $U$ is a symmetry
(i.e., a unitary involution) such as
   \begin{align*}
U= \begin{bmatrix} 1 & 0 \cr 0 & -1
\cr\end{bmatrix}=1\oplus(-1).
   \end{align*}

According to Proposition~\ref{C3.7}, the following
holds.
   \begin{align*}
   \begin{minipage}{65ex}
{\em The power sequence of a unitary operator with no
identity part is weakly convergent if and only if it
is weakly stable.}
   \end{minipage}
   \end{align*}

The weak stability of a unitary operator $U$ can be
completely characterized by the requirement that the
spectral measure of $U$ is Rajchman (see
Corollary~\ref{undsw}; see also \eqref{(3)}). Let us
also note that, since unitary operators are clearly
not strongly stable (and therefore not uniformly
stable), it follows from Corollary~\ref{C.1Z} that the
only unitary operator whose power sequence is strongly
or norm convergent is the identity operator.

Let $\lambda$ and $\mu$ be $\sigma$-finite measures on
the $\sigma$-algebra $\A_\TT$ of Borel subsets of the
unit circle $\TT$ centered at the origin of the
complex plane $\CC$. The Lebesgue decomposition
theorem implies that the measure $\mu$ has a unique
decomposition $\mu=\mu_a+\mu_s$ relative to $\lambda$,
where $\mu_a$ and $\mu_s$ are measures on $\A_\TT$
that are absolutely continuous and singular relative
to $\lambda$, respectively. In turn, the measure
$\mu_s$ has a unique decomposition
$\mu_s=\mu_{sc}+\mu_{sd}$, where $\mu_{sc}$ and
$\mu_{sd}$ are measures on $\A_\TT$ that are
continuous and discrete (i.e., pure point),
respectively (cf.\ \cite[Theorem~I.13]{Re-Si72}). Call
the measures $\mu_{sc}$ and $\mu_{sd}$ {\em
singular-continuous} and {\em singular-discrete}
relative to $\lambda$, respectively (cf.\
\cite[Proposition~7.13]{Kub15}).

A unitary operator is {\it absolutely continuous},
{\it singular-continuous}, or {\it singular-discrete}
if its spectral measure is absolutely continuous,
singular-continuous, or singular-discrete relative to
the normalized Lebesgue measure on $\A_\TT,$
respectively. By the spectral theorem, every unitary
operator $U\in \B(\H)$ on a Hilbert space $\H$
decomposes as the orthogonal sum
   \begin{align} \label{(2)}
U=U_{a}\oplus U_{sc} \oplus U_{sd}
   \end{align}
of unitary operators relative to the orthogonal
decomposition $\H=\H_a \oplus \H_{sc} \oplus \H_{sd}$
of Hilbert spaces, where $U_{a}\in \B(\H_a)$ is
absolutely continuous, $U_{sc}\in \B(\H_{sc})$ is
singular-continuous and $U_{sd}\in \B(\H_{sd})$ is
singular-discrete. Note that any part in the
decomposition \eqref{(2)} may be absent.

The weak stability of unitary operators can be
characterised in terms of the decomposition
\eqref{(2)} as follows.
   \begin{remark} \label{Rem.1}
Consider the decomposition \eqref{(2)} of a unitary
operator $U\in \B(\H)$ relative to the decomposition
$\H = \H_a \oplus \H_{sc} \oplus \H_{sd}$. The
following facts are well known (see, e.g., \cite[p.\
48]{Kub2}).
   \begin{enumerate}
   \item[(i)] {\it An absolutely continuous unitary is always weakly
stable, i.e., $U_{a}^n\wconv 0$.}
   \item[(ii)] {\it A singular-discrete unitary is never weakly
stable, i.e., $U_{sd}^n\notwconv 0$.}
   \end{enumerate}
This in turn is equivalent to the following statement.
   \begin{align} \label{eyivc}
   \begin{minipage}{71ex}
{\em A unitary operator is weakly stable if and only
if its singular-continu\-ous part is weakly stable and
its singular-discrete part is absent.}
   \end{minipage}
   \end{align}
Recall also the following two observations from
\cite[Propositions~3.2 and 3.3]{Kub2}.
   \begin{enumerate}
   \item[(iii)] {\it There exist weakly stable singular-continuous
unitary operators}.
   \item[(iv)] {\it There exist weakly unstable
singular-continuous unitary operators}.
   \end{enumerate}
Now, we provide more details.

(i$^\prime$) {\sf Absolutely continuous unitaries as
parts of bilateral shifts}. A unitary operator is
absolutely continuous if and only if it is a part of a
bilateral shift (see \cite[p.\ 56, Exercise~8]{Fil}).

(ii$^\prime$) {\sf Singular-discrete unitaries via
eigenvalues}. A complex number $\alpha$ is an
eigenvalue of a unitary operator $U$ with the spectral
measure $E$ if and only if $E(\{\alpha\})\ne 0$ (see
\cite[Theorem~12.29]{Rud}). Hence, $U$ has no
singular-discrete (equivalently, discrete) part if and
only if it has an eigenvalue.

(ii$^{\prime\prime}$) {\sf Weakly unstable
singular-discrete unitaries with full spectrum}. If
$\{\alpha_k\}_{k=1}^{\infty}$ is an enumeration of the
rationals in $[0,1)$, then the diagonal operator $U$
with diagonal $\{e^{2\pi i\alpha_k}\}_{k=1}^{\infty}$
on $\ell_+^2$ is a singular-discrete unitary whose
spectrum is the whole unit circle $\TT$, and which is
not weakly stable according to (ii) (see also
\cite[Example~13.5]{Dow}).

(iii$^\prime$) {\sf Weakly stable singular-continuous
unitaries}. Let $\mu$ be a finite measure on $\A_\TT$
and let $L^2(\TT,\mu)$ be the Hilbert space of square
integrable Borel complex functions on $\TT$ with
respect to $\mu$. Consider the unitary multiplication
operator $U_{\vphi,\mu}$ on $L^2(\TT,\mu)$, which is
defined by
$$
U_{\vphi,\mu} \psi = \vphi \cdot \psi \;\;
\textrm{a.e.\ [$\mu$]}, \quad \psi\in L^2(\TT,\mu),
$$
where $\vphi\colon \TT\to\TT$ is the identity map.
Note that the measure $\mu$ can be regarded as the
scalar spectral measure of $U_{\vphi,\mu}.$ Recall
that a finite measure $\nu$ on $\A_\TT$ is a {\em
Rajchman measure} if $\int_\TT z^k\,d\nu(z)\to 0$ as
$|k|\to\infty$ (equivalently as $k\to\infty$). For the
operator $U_{\vphi,\mu}$, we have
   \begin{align} \label{(3)}
\textit{$\mu$ is a Rajchman measure if and only if
$U_{\vphi,\mu}^{n}\wconv 0$.}
   \end{align}
(Indeed, $U_{\vphi,\mu}^{n} \wconv 0$ $\limply$
$\<U_{\vphi,\mu}^{n}1,1\>\to 0$ $\iff$ $\int_\TT
z^n\,d\mu(z) \to 0$ $\limply$
${U_{\vphi,\mu}^{n}\wconv 0}$; the proof of the last
implication is straightforward (cf.\ \cite[pp.\
1383/1384]{BM}). A Rajchman measure is always
continuous (see \cite{Ned20}, see also \cite[p.\
364]{Lyo}); however there are singular Rajchman
measures \cite[Theorem~3.4]{Lyo}. Thus a singular
Rajchman measure is singular-continuous, so, by
\eqref{(3)}, we have
   $$
   \begin{minipage}{78ex}
{\it if $\mu$ is a singular Rajchman meas\-ure, then
the multiplication operator $U_{\vphi,\mu}$ on
$L^2(\TT,\mu)$ is a weakly stable singular-continuous
unitary}.
   \end{minipage}
   $$

(iv$^\prime$) {\sf Weakly unstable singular-continuous
unitaries}. Let $\mu$ be the Borel-Stieltjes measure
on $\A_\TT$ generated by the Cantor function with
domain being the Cantor set $\varGamma$ over the unit
circle $\TT$ (cf.\ \cite[p.\ 20, Example~3]{Re-Si72}
and \cite[p.\ 128, Problem~7.15(c)]{Kub15}). Then
$\mu$ is singular-continuous, and thus the
multiplication operator $U_{\vphi,\mu}$ on
$L^2(\TT,\mu)$ is a singular-continuous unitary whose
spectrum is ${\varGamma}$, but $\mu$ is not a Rajchman
measure (see \cite[p.\ 364]{Lyo}), so by \eqref{(3)},
the operator $U_{\vphi,\mu}$ is not weakly stable.
   \phantom{} \hfill $\diamondsuit$
   \end{remark}
Now, we apply Proposition~\ref{C3.7} to the
decomposition \eqref{(2)} of a unitary operator $U$
and conclude that the power sequence
$\{U^n\}_{n=1}^{\infty}$ is weakly convergent if and
only if its singular-continuous part $U_{sc}$ is
weakly stable and $U_{sd}=I$ (see Theorem~\ref{T2}).
   \begin{lemma}  \label{lim4.3}
The power sequence of a singular-continuous unitary
operator $U$ converges weakly if and only if $U$ is
weakly stable.
   \end{lemma}
   \begin{proof}
Only the necessity part requires proof. Decompose $U$
as in Proposition~\ref{C3.7}, that is $U=I\oplus W.$
Since $U$ is singular-continuous, $U$ has no
eigenvalue (see Remark~\ref{Rem.1}(ii$^\prime$)), and
thus $U=W$. Using the ``moreover'' part of
Proposition~\ref{C3.7} completes the proof.
   \end{proof}
   \begin{lemma} \label{lim4.4}
The power sequence of a singular-discrete unitary
operator $U$ converges weakly if and only if $U$ is
the identity operator.
   \end{lemma}
   \begin{proof}
As above, we only discuss the necessity part.
Decompose $U$ as in Proposition~\ref{C3.7}. Since $U$
is singular-discrete, so is $W$. Therefore, by the
``moreover'' part of Proposition~\ref{C3.7} and
Remark~\ref{Rem.1}(ii), $U=I$.
   \end{proof}
   \begin{theorem} \label{T2}
Let $U$ be a unitary operator on a Hilbert space $\H$
and let $U=U_a\oplus U_{sc}\oplus U_{sd}$ be the
decomposition as in \eqref{(2)}. Then
$\{U^n\}_{n=1}^{\infty}$ is weakly convergent if and
only if $U_{sc}$ is weakly stable and $U_{sd}$ is the
identity operator $($any of $U_a$, $U_{sc}$ and
$U_{sd}$ may be absent\/$)$. Moreover, if this is the
case, then the weak limit of $\{U^n\}_{n=1}^{\infty}$
is the orthogonal projection of $\H$ onto
$\N(I-U)=\H_{sd}$.
   \end{theorem}
   \begin{proof}
By \eqref{(1)}, $\{U^n\}_{n=1}^{\infty}$ converges
weakly if and only if every sequence
$\{U_a^{n}\}_{n=1}^{\infty}$,
$\{U_{sc}^{n}\}_{n=1}^{\infty}$ and
$\{U_{sd}^{n}\}_{n=1}^{\infty}$ converges weakly. It
follows from Remark~\ref{Rem.1}(i) that $U_a^{n}
\wconv 0$ as $n\to\infty$. Now, using
Lemmas~\ref{lim4.3} and \ref{lim4.4}, we can prove the
``if and only if'' part. The ``moreover'' part follows
from the ``if and only if'' part and
Proposition~\ref{C3.7}.
   \end{proof}
Observe that Theorem~\ref{T2} extends the weak
stability criterion \eqref{eyivc}.
\section{\label{Sec.5}The case of subnormal operators}
In this section, we characterize the weak, strong and
uniform stability of subnormal operators in terms of
their semispectral measures. As a consequence, we
obtain results on the convergence of the power
sequence of a subnormal operator with respect to the
weak, strong and norm topologies.

We begin by reviewing the concepts of the semispectral
integral and the semispectral measure of a subnormal
operator. Suppose that $F\colon \acal\to \B(\H)$ is a
semispectral measure on a $\sigma$-algebra $\acal$ of
subsets of a set $\varOmega$, i.e.,
$\mu_x=\<F(\cdot)x,x\>$ is a positive measure for
every $x\in\H$, and $F(\varOmega)=I.$ Also recall that
if $\varphi\in\bigcap_{x\in\H}L^1(\mu_x)$, then there
exists a unique operator in $\B(\H)$, denoted by
$\int_\varOmega\varphi(\omega)\,F(\D \omega)$, such
that
   \begin{align} \label{(3.5)}
\Big\<\int_\varOmega\varphi(\omega)\,F(\D
\omega)x,x\Big\> =\int_\varOmega\varphi(\omega)\<F(\D
\omega)x,x\>, \quad x\in\H.
   \end{align}
The same applies to positive operator valued measures.

If $T\in \B(\H)$ is a subnormal operator and $E$ is
the spectral measure of a minimal normal extension
$N\in\B(\K)$ of $T$, then the Borel semispectral
measure $F$ on $\CC$ with values in $\B(\H)$ defined
by
   \begin{align} \label{(3.8)}
F(\varDelta)=P_\H E(\varDelta)|_\H, \quad
\varDelta\text{ - Borel subset of }\CC,
   \end{align}
where $P_{\H}\in \B(\K)$ is the orthogonal projection
of $\K$ onto $\H$, is called the {\it semispectral
measure}\/ of $T.$ The definition of $F$ is
independent of the choice of a minimal normal
extension of $T$ and $F$ is a unique representing
semispectral measure of the operator valued complex
moment sequence $\{T^{*n}T^m\}_{m,n=0}^{\infty}$,
i.e.,
   \begin{align} \label{(4)}
T^{*n}T^m=\int_{\CC}z^m \bar z^n F(\D z), \quad m,n\Ge 0.
   \end{align}
The closed support of $F$ is equal to $\sigma(N).$ See
\cite[Appendix]{St92} and \cite[Section~3]{Ju-St08} (see also
\cite{Bi-So87,Co91}) for more information.

If $F$ is a $\B(\H)$-valued Borel semispectral measure
on $\CC$, then $F_{\TT}$ denotes the positive operator
valued measure defined as $F_{\TT}(\varDelta) =
F(\varDelta \cap \TT)$ for any Borel subset of $\CC$.

We begin by discussing the weak stability of subnormal
operators.
   \begin{proposition} \label{wykstu}
Let $T\in \B(\H)$ be a subnormal operator with the
semispectral measure $F$. Then $T$ is weakly stable if
and only if $\|T\| \Leq 1$ and $\<F_{\TT}(\cdot)x,x\>$
is a Rajchman measure for every $x\in \H$, that is,
   \begin{align}  \label{ntift}
\lim_{n\to \infty} \int_{\TT} z^n \<F(\D z)x,x\> = 0,
\quad x\in \H.
   \end{align}
   \end{proposition}
   \begin{proof}
If $T$ is weakly stable, then by \eqref{gwpi} and
\eqref{shnyr}, $\|T\| \Leq 1$. Hence, $\sigma(T)
\subseteq \DD \sqcup\TT$ and thus, by \eqref{(3.5)}
and \eqref{(4)}, the following identity holds
   \begin{align} \label{tnhh}
\<T^n x,x\> = \int_{\DD} z^n \<F(\D z)x,x\> +
\int_{\TT} z^n \<F(\D z)x,x\>, \quad x\in \H, \, n
\Geq 0.
   \end{align}
By the Lebesgue dominated convergence theorem, the first
summand in \eqref{tnhh} tends to zero as $n\to \infty$.
Therefore, $T$ is weakly stable if and only if \eqref{ntift}
holds.
   \end{proof}
The following result is a direct consequence of
Proposition~\ref{wykstu} and the obvious observation
that semispectral measures of normal (and therefore
unitary) operators are spectral. It can also be
derived from \eqref{(3)} and two facts, the first of
which says that any unitary operator is an orthogonal
sum (of arbitrary cardinality) of unitary
multiplication operators, as in
Remark~\ref{Rem.1}(iii$^\prime$), and the second
states that a finite Borel measure on $\TT$, which is
the sum of a series of any (not necessarily countable)
family of Rajchman measures, is a Rajchman measure.
   \begin{corollary} \label{undsw}
A unitary operator $U\in \B(\H)$ with the spectral
measure $E$ is weakly stable if and only if
$\<E(\cdot)x,x\>$ is a Rajchman measure for every
$x\in \H$.
   \end{corollary}
Since absolutely continuous measures with respect to
the normalized Lebesgue measure on $\A_\TT$ are
Rajchman measures (see \cite[p.\ 364]{Lyo}), we get.
   \begin{corollary}
Let $T\in \B(\H)$ be a subnormal contraction with the
semispectral measure $F$. If $F_{\TT}$ is absolutely
continuous with respect to the normalized Lebesgue
measure on $\A_\TT$, then $T$ is weakly stable. In
particular, this is the case if $F_{\TT}=0$.
   \end{corollary}
Using Corollary~\ref{C.1Z} and
Proposition~\ref{wykstu} (see also \eqref{(1)}), we
can characterize the weak convergence of the power
sequence of a subnormal operator as follows.
   \begin{proposition} \label{traq}
Let $T\in \B(\H)$ be a subnormal operator. Then
$\{T^n\}_{n=1}^{\infty}$ is weakly convergent if and
only if $T$ decomposes as $T=I\oplus L$ with respect
to a decomposition $\H=\H_1 \oplus \H_2$, where $L$ is
a subnormal contraction with the semispectral measure
$G$ such that $\<G_{\TT}(\cdot)x,x\>$ is a Rajchman
measure for every $x\in \H_2.$ If this is the case,
then the weak limit of $\{T^n\}_{n=1}^{\infty}$ is the
orthogonal projection of $\H$ onto $\N(I-T)$.
Moreover, the orthogonal decomposition of $T$ is
unique and $\H_1=\N(I-T).$
   \end{proposition}
As for the strong stability of subnormal operators, we
have the following result.
   \begin{proposition}[\mbox{\cite[Proposition~4.2(ii)]{JJS23}}]
\label{sytsrb} Let $T\in \B(\H)$ be a subnormal
operator with the semispectral measure $F$. Then $T$
is strongly stable if and only if $\|T\|\Leq 1$ and
$F(\TT)=0$, or equivalently, if and only if $T$ is
power bounded and $F(\TT)=0$.
   \end{proposition}
It turns out that weakly stable normal operators are
simply orthogonal sums of weakly stable unitary
operators and strongly stable normal operators. In
fact, as the proof of Corollary~\ref{nyrnil} shows,
any normal contraction $T$ has the form $T=U \oplus
S$, where $U$ is unitary and $S$ is a strongly stable
subnormal contraction. This kind of orthogonal
decompositions for a class of contractions, covering
the case of normal operators, can be found in
\cite{Kub96}.
   \begin{corollary} \label{nyrnil}
A normal operator $T\in \B(\H)$ is weakly stable if
and only if $T=U\oplus S$, where $U$ is a weakly
stable unitary operator and $S$ is a strongly stable
normal operator.
   \end{corollary}
   \begin{proof}
First, observe that the semispectral measure $E$ of
$T$ is the spectral measure of $T$. Suppose that $T$
is weakly stable. According to
Proposition~\ref{wykstu}, $\|T\| \Leq 1$~and
   \begin{align} \label{rydn}
\textit{$\<E_{\TT}(\cdot)x,x\>$ is a Rajchman measure
for every $x\in \H$.}
   \end{align}
By \cite[Theorem~6.6.3]{Bi-So87}, the spaces
$\M:=\R(E(\TT))$ and $\N:=\R(E(\DD))$ reduce $T$,
$\H=\M\oplus \N$ and $T=U\oplus S$, where $U=T|_{\M}$
and $S=T|_{\N}$. Clearly, $U$ is unitary (see
\cite[Theorem~6.1.2]{Bi-So87}) and $S$ is a normal
contraction. It is easy to see that $E_{\TT}$ is the
spectral measure of $U$ and the measure $E_{\DD}$
defined by $E_{\DD}(\varDelta)=E(\varDelta \cap \DD)$
for any Borel subset $\varDelta$ of $\cbb$ is the
spectral measure of $S$. Hence, by \eqref{rydn} and
Corollary~\ref{undsw}, $U$ is weakly stable. Since
$E_{\DD}(\TT)=0$, we deduce from
Proposition~\ref{sytsrb} that $S$ is strongly stable.
In view of \eqref{(1)}, the converse implication is
obvious.
   \end{proof}
Before we deal with the strong convergence of power
sequences of subnormal operators, we prove the
following.
   \begin{lemma} \label{pygut}
Let $F$ be a $\B(\H)$-valued Borel semispectral
measure on $\CC$. Then
   \begin{align} \label{gwqa}
\text{$\left\{\int_{\DD} z^n
F(dz)\right\}_{n=1}^{\infty}$ converges strongly to
$0$.}
   \end{align}
   \end{lemma}
   \begin{proof}
By Naimark's dilation theorem (see \cite{Nai43} and
\cite[Theorem~6.4]{Ml78})), there exist a Hilbert space $\K$
and a spectral measure $E\colon \A_\CC \to \B(\K)$ such that
$\H \subseteq \K$ and
   \begin{align*}
\<F(\varDelta)x,y\> = \<E(\varDelta)x,y\>, \quad x,y
\in \H,\, \varDelta\in \A_\CC,
   \end{align*}
where $\A_\CC$ stands for the $\sigma$-algebra of all
Borel subsets of $\CC$. This, together with
\eqref{(3.5)}, implies that
   \allowdisplaybreaks
   \begin{align*}
\Big|\Big\<\int_{\DD} z^n F(dz)x,y\Big\>\Big| &=
\Big|\int_{\DD} z^n \< E(dz)x,y\>\Big|
   \\
& = \Big|\Big\<\int_{\DD} z^n E(dz)x,y\Big\>\Big|
   \\
& \Leq \Big\|\int_{\DD} z^n E(dz)x\Big\| \|y\|
   \\
& = \Big(\int_{\DD} |z|^{2n} \< F(dz)x,x\>\Big)^{1/2}
\|y\|, \quad x, y\in \H, \, n \Ge 0.
   \end{align*}
As a consequence, we have
   \begin{align*}
\Big\|\int_{\DD} z^n F(dz)x\Big\| \Leq \Big(\int_{\DD}
|z|^{2n} \< F(dz)x,x\>\Big)^{1/2}, \quad x \in \H, \,
n \Ge 0.
   \end{align*}
By the Lebesgue dominated convergence theorem,
$\int_{\DD} |z|^{2n} \< F(dz)x,x\> \to 0$ as $n\to
\infty$, so \eqref{gwqa} holds.
   \end{proof}
Now we can characterize the strong convergence of
power sequences of subnormal operators.
   \begin{theorem} \label{swcti}
Let $T\in \B(\H)$ be a subnormal operator with the
semispectral measure $F$. Then the following conditions are
equivalent{\em :}
   \begin{enumerate}
   \item[(i)] $\{T^n\}_{n=1}^{\infty}$ is strongly
convergent,
   \item[(ii)] $T$ decomposes as $T=I\oplus L$ with
respect to a decomposition $\H=\H_1 \oplus \H_2$,
where $L$ is a subnormal contraction whose
semispectral measure vanishes on $\TT$,
   \item[(iii)] $T$ is a contraction and
   \begin{align} \label{fdyl}
\text{$F(\varDelta)=\delta_{1}(\varDelta) F(\TT)$ for every
Borel subset $\varDelta$ of $\TT$,}
   \end{align}
where $\delta_1$ is the Dirac measure at $1$.
   \end{enumerate}
Moreover, if {\em (i)} holds, then $F(\TT)$ is the
orthogonal projection of $\H$ onto \mbox{$\N(I-T)$}
and the strong limit of $\{T^n\}_{n=1}^{\infty}$ is
equal to $F(\TT)$. Furthermore, the orthogonal
decomposition of $T$ in {\em (ii)} is unique and
$\H_1=\N(I-T)$.
   \end{theorem}
   \begin{proof}
There is no loss of generality in assuming that $\|T\|\Leq 1$
(see the proof of Proposition~\ref{wykstu}).

(i)$\Rightarrow$(ii) It follows from
Corollary~\ref{C.1Z} that $T$ decomposes as $T=I\oplus
L$ with respect to the orthogonal decomposition $\H =
\N(I-T) \oplus \N(I-T)^\perp$, where $L$ is a
subnormal contraction. Hence, if
$\{T^n\}_{n=1}^{\infty}$ is strongly convergent, then
by Corollary~\ref{C.1Z}, $L$ is strongly stable. Now
we can apply Proposition~\ref{sytsrb} to get (ii).

(ii)$\Rightarrow$(iii) Denote by $G$ the semispectral measure
of $L$. Then (see \cite[Lemma~3.1]{CJJS2})
   \begin{align} \label{fdila}
F(\varDelta)= \delta_1(\varDelta) I \oplus
G(\varDelta), \quad \varDelta \text{ - Borel subset of
} \CC.
   \end{align}
Substituting $\varDelta=\TT$ and using $G(\TT)=0$, we
get $F(\TT)= I \oplus 0$. Since $G(\TT)=0$ implies
that $G(\varDelta)=0$ for every Borel subset
$\varDelta$ of $\TT$, we see that \eqref{fdila}
implies \eqref{fdyl}, so (iii) holds. According to
Proposition~\ref{sytsrb}, the subnormal contraction
$L$ is strongly stable, so $\{T^n\}_{n=1}^{\infty}$
converges strongly to $I \oplus 0$. In turn, by
Corollary~\ref{C.1Z}, $I \oplus 0=F(\TT)$ is the
orthogonal projection of $\H$ onto $\N({I-T})$. This
implies that $\H_1=\N({I-T})$, and consequently shows
that the orthogonal decomposition of $T$ in (ii) is
unique. In summary, we have proven both the
``moreover'' and ``furthermore'' parts.

(iii)$\Rightarrow$(i) It follows from the inclusion
$\sigma(T) \subseteq \DD \sqcup\TT$ that
   \allowdisplaybreaks
   \begin{align}  \notag
T^n & \overset{\eqref{(4)}} = \int_{\DD} z^n F(\D z) +
\int_{\TT} z^n F(\D z)
   \\ \label{piroT}
& \overset{\eqref{fdyl}}= \int_{\DD} z^n F(\D z) + F(\TT),
\quad n \Geq 0.
   \end{align}
Therefore, by Lemma~\ref{pygut}, $\{T^n\}_{n=1}^{\infty}$
converges strongly to $F(\TT)$.
   \end{proof}
   \begin{remark}
It is possible to prove the implication
(iii)$\Rightarrow$(ii) of Theorem~\ref{swcti} directly
without using Lemma~\ref{pygut}. We are still assuming
that $\|T\|\Leq 1$. Let $T=I\oplus L$ be the
orthogonal decomposition of $T$ as in
Corollary~\ref{C.1Z}. According to \eqref{piroT}, we
have
   \begin{align} \label{piesz}
T^n = \int_{\DD} z^n F(\D z) + F(\TT), \quad n \Geq 0.
   \end{align}
It follows from \eqref{(3.5)} and the Lebesgue
dominated convergence theorem that the sequence
$\big\{\int_{\DD} z^n F(\D z)\big\}_{n=1}^{\infty}$
converges weakly to $0$ as $n\to \infty$ (this is less
than what is postulated in \eqref{gwqa}). Hence, by
\eqref{piesz}, the power sequence
$\{T^n\}_{n=1}^{\infty}$ of $T$ converges weakly to
$F(\TT)$. In view of Corollary~\ref{C.1Z}, $F(\TT)$ is
the orthogonal projection of $\H$ onto $\N({I-T})$.
Let $G$ be the semispectral measure of $L$. Then
   \begin{align*}
0 \overset{\eqref{fdyl}} = F(\TT\setminus \{1\})
\overset{\eqref{fdila}}= 0\oplus G(\TT\setminus \{1\}),
   \end{align*}
so $G_{\TT}(\varDelta) = \delta_1(\varDelta) G(\{1\})$
for all Borel subsets $\varDelta$ of $\TT$. However,
by Proposition~\ref{traq}, $\<G_{\TT}(\cdot)x,x\>$ is
a Rajchman measure for every $x\in \H$, so
$\<G(\TT)x,x\>=0$ for all $x\in \H$ or, equivalently,
$G(\TT)=0$. This yields (ii).
   \hfill $\diamondsuit$
   \end{remark}
In the last part of the paper, we will focus on the
uniform stability and the associated norm convergence
of power sequences of subnormal operators. It follows
from \eqref{ustypiy} and \eqref{shnyr} that
   \begin{align} \label{(4.5)}
   \begin{minipage}{55ex}
   {\em \it a subnormal operator is uniformly stable if and
only if it is a strict contraction}.
   \end{minipage}
   \end{align}
First, we show that the uniform stability of subnormal
operators can be characterized in terms of strong
stability as follows.
   \begin{theorem}
Let $T\in \B(\H)$ be a subnormal operator with the
semispectral measure $F$. Then the following statements are
equivalent{\em :}
   \begin{enumerate}
   \item[(i)]
$T$ is uniformly stable,
   \item[(ii)]
$T$ is strongly stable and $\int_\DD
\frac{1}{1-|z|^2}\<F(\D z)x,x\> < \infty$ for every
$x\in \H$,
   \item[(iii)]
$r(T) \Leq 1$, $F(\TT)=0$ and $\int_\DD
\frac{1}{1-|z|^2}\<F(\D z)x,x\> <\infty$ for every
$x\in \H$.
   \end{enumerate}
Moreover, if $T$ is uniformly stable, ${N\in \B(\K)}$
is a minimal normal extension of\/ $T$ and\/ $P_{\H}$
in $\B(\K)$ is the orthogonal projection of\/ $\K$
onto\/ $\H$, then ${\|N\|<}1$ and
   \begin{align} \label{(5)}
\sum_{j=0}^{\infty} T^{*j}T^j = \int_\DD
\frac{1}{1-|z|^2} F(\D z) = P_{\H}(I-N^*N)^{-1}|_{\H},
   \end{align}
where the series is norm convergent.
   \end{theorem}
   \begin{proof}
Let $E$ be the spectral measure of $N$. Using
\eqref{shnyr} and \cite[Corollary~II.2.17]{Co91}, we
deduce that
   \begin{align} \label{(6)}
   \begin{minipage}{58ex}
{\em $\theta:=\|N\|=\|T\| = r(T)$ and both measures $E$ and
$F$ are supported in $\{z\in \CC\colon |z|\Leq \theta\}$}.
   \end{minipage}
   \end{align}

(i)$\Rightarrow$(iii) This is a direct consequence of
\eqref{(4.5)}.

(iii)$\Rightarrow$(i) By \eqref{(6)}, $T^*T \Leq I.$
Since $\int_\DD\frac{1}{1-|z|^2}\<F(\D z)x,x\><\infty$
for every ${x\in \H}$, there exists the semispectral
integral $R:=\int_\DD \frac{1}{1-|z|^2} F(\D z)\in
\B(\H)$ (see \cite[Theorem~A.1]{St92}). It follows
from (iii) that
   \begin{align} \label{(7)}
F(\CC\setminus\DD)=0.
   \end{align}
Hence, we have
   \begin{align*}
\<Rx, x\> \overset{\eqref{(3.5)}}
=\int_\DD\frac{1}{1-|z|^2}\<F(\D z)x,x\>\Ge\|x\|^2,
\quad x\in \H,
   \end{align*}
and thus $R \neq 0$. Using the Cauchy-Schwarz
inequality, \eqref{(3.5)} and \eqref{(4)}, we obtain
   \allowdisplaybreaks
   \begin{align*}
\|x\|^2 & \overset{\eqref{(7)}} = \int_\DD 1 \, \<F(\D
z)x,x\>
   \\
& \hspace{1.4ex}\Leq \bigg(\int_\DD (1-|z|^2) \<F(\D
z)x,x\>\bigg)^{1/2} \; \bigg(\int_\DD
\frac{1}{1-|z|^2}\<F(\D z)x,x\>\bigg)^{1/2}
   \\
& \hspace{1.4ex} = \<(I-T^*T)x,x\>^{1/2}\<Rx,x\>^{1/2}
   \\
& \hspace{1.4ex} \Leq \|(I-T^*T)^{1/2}x\| \|R\|^{1/2}
\, \|x\|, \quad x\in \H.
    \end{align*}
Therefore, we have
         \begin{align*}
\|(I-T^*T)^{1/2}x\| \Ge \frac{1}{\|R\|^{1/2}} \|x\|,
\quad x\in \H.
         \end{align*}
This implies that ${(I-T^*T)^{1/2}}$ is invertible.
Thus, ${0\notin\sigma(I-T^*T)}$, or equivalently
${1\notin \sigma(T^*T)}$. Since ${T^*T\Leq I}$, we
conclude that ${\|T\|<1}$. Hence, in view of
\eqref{(4.5)}, statement (i) holds.

(ii)$\Leftrightarrow$(iii) This equivalence is a
direct consequence of Proposition~\ref{sytsrb} and
\eqref{shnyr}.

It remains to prove the ``moreover'' part. Assume that
$T$ is uniformly stable. By \eqref{(4.5)},
${\|T\|<1}$, so the series ${\sum_{j=0}^{\infty}
T^{*j}T^j}$ is norm convergent. Using \eqref{(4)}, we
get
   \allowdisplaybreaks
   \begin{align*}
\bigg\<\bigg(\sum_{j=0}^\infty T^{*j}T^j \bigg)x,x
\bigg\> & = \sum_{j=0}^{\infty} \int_\DD
|z|^{2j}\<F(\D z)x,x\>
   \\
& = \int_\DD \bigg(\sum_{j=0}^{\infty} |z|^{2j}\bigg)
\<F(\D z)x,x\>
   \\
&\hspace{-1ex} \overset{\eqref{(3.5)}}= \bigg\<
\int_\DD \frac{1} {1-|z|^2} F(\D z)x,x\bigg\>, \quad
x\in \H.
      \end{align*}
This proves the first equality in \eqref{(5)}. By
\eqref{(6)}, the operator ${I-N^*N}$ is invertible.
Hence, by \eqref{(3.5)}, \eqref{(3.8)}, \eqref{(6)}
and the Stone-von Neumann calculus, we~have
   \allowdisplaybreaks
   \begin{align*}
\Big<\int_\DD \frac{1}{1-|z|^2} F(\D z)x,x\Big\> & =
\Big\<\int_\DD\frac{1}{1-|z|^2} E(\D z)x,x\Big\>
   \\
& = \<(I-N^*N)^{-1}x,x\>
   \\
& =\<P_{\H}(I-N^*N)^{-1}|_{\H}x,x\>, \quad x\in \H.
         \end{align*}
This proves the second equality in \eqref{(5)} and
completes the proof.
   \end{proof}
Now, we can describe the norm convergence of power
sequences of subnormal operators. The result below is
a direct consequence of \eqref{(4.5)} and
Corollary~\ref{C.1Z}.
   \begin{proposition}
Let $T\in \B(\H)$ be a subnormal operator. Then
$\{T^n\}_{n=1}^{\infty}$ is norm convergent if and
only if $T$ decomposes as $T=I\oplus L$, where $L$ is
a strict contraction.
   \end{proposition}
We conclude the paper by relating the stability of a
subnormal operator to the stability of its minimal
normal extension.
   \begin{theorem}
Let $T\in \B(\H)$ be a subnormal operator and
$N\in\B(\K)$ be a minimal normal extension of $T$.
Then the following assertions hold{\em :}
   \begin{enumerate}
   \item[(i)] $T$ is weakly $($strongly, uniformly$)$ stable if and
only if $N$ is weakly $($strongly, uniformly$)$
stable,
   \item[(ii)] $\{T^n\}_{n=1}^{\infty}$ converges weakly $($strongly,
in norm$)$ if and only if $\{N^n\}_{n=1}^{\infty}$
converges weakly $($strongly, in norm$)$,
   \item[(iii)] if
$\{T^n\}_{n=1}^{\infty}$ converges weakly $($strongly,
in norm$)$, then $\N(I-T)=\N(I-N)$, $T=I\oplus L$
relative to $\H=\N(I-T) \oplus \N(I-T)^{\perp}$ and
$N=I\oplus M$ relative to $\K=\N(I-N) \oplus
\N(I-N)^{\perp}$, where $L$ and $M$ are weakly
$($strongly, uniformly$)$ stable and $M$ is a minimal
normal extension of $L$.
   \end{enumerate}
   \end{theorem}
   \begin{proof}
(i) First, we will address the issue of weak
stability. Suppose that $T$ is weakly stable. Then
$\|N\|=\|T\| \Leq 1$ (see \eqref{(6)} and
Proposition~\ref{wykstu}). It follows from
\cite[Proposition~II.2.4]{Co91} that
   \begin{align} \label{mnte}
\K=\bigvee_{j=0}^{\infty} N^{*j}\H.
   \end{align}
Then for all $j,k\Ge 0$ and $n \Geq 1$,
   \begin{align*}
\<N^n (N^{*j}x), N^{*k}y\> = \<N^{n+k}x, N^j y\> =
\<T^n (T^k x), T^jy\>, \quad x,y \in \H.
   \end{align*}
By the weak stability of $T$, this implies that $\<N^n
f, g\> \to 0$ as $n\to \infty$ for all $f, g\in \X$,
where $\X$ is the linear span of
$\bigcup_{j=0}^{\infty} N^{*j}\H$. Since by
\eqref{mnte}, $\X$ is dense in $\K$, and $\sup_{n\Geq
1} \|N^n\| < \infty$, we conclude that $N$ is weakly
stable (see e.g., \cite[Lemma~1]{Kub89}). The converse
implication is obvious.

The case of strong stability can be considered
similarly (see also
\cite[Proposition~4.2(iii)]{JJS23}). In turn, the case
of uniform stability follows from the fact that
$\|T^n\|=\|N^n\|$ for all $n\Ge 1$ (see
\cite[Corollary~II.2.17]{Co91}).

(ii)\&(iii) Assume that $\{T^n\}_{n=1}^{\infty}$
converges weakly to $P\in \B(\H)$. Then, by
Corollary~\ref{C.1Z}, $P$ is an orthogonal projection
with $\R(P)=\N(I-T)$ and $T$ decomposes as $T=I \oplus
L$ with respect to the orthogonal decomposition $\H =
\R(P) \oplus \R(I-P)$, where $L$ is a weakly stable
subnormal operator. It follows from
\cite[Lemma~3.4]{CJJS23} that $N=I \oplus M$ relative
to $\K=\R(P) \oplus \R(P)^{\perp}$, where $M$ is a
minimal normal extension of $L$. By (i), $M$ is weakly
stable, so $\{N^n\}_{n=1}^{\infty}$ convergent weakly
to $Q\in \B(\K)$, where $Q$ is the orthogonal
projection of $\K$ onto $\R(P)$ (clearly, $Q$ extends
$P$). Applying Corollary~\ref{C.1Z} to $N$ in place of
$T$, we conclude that
   \begin{align*}
\N(I-N)=\R(Q)=\R(P)=\N(I-T).
   \end{align*}
The ``if'' part of (ii) for the weak topology is again
obvious. This proves both (ii) and (iii) in the case
of weak topology. A similar argument can be used in
the case of strong and norm topologies.
   \end{proof}
\bibliographystyle{amsplain}

\begin{thebibliography}{10}

\bibitem{Ag-St95-6} J. Agler, M. Stankus, $m$-isometric
transformations of Hilbert spaces, I, II, III, {\it
Integr. Equ. Oper. Theory} {\bf 21, 23, 24} (1995,
1995, 1996), 383-429, 1-48, 379-421.

\bibitem{BM}
C. Badea, V. M\"uller, On weak orbits of operators,
{\em Topology Appl.} {\bf 156} (2009), 1381-1385.


   \bibitem{Bi-So87} M. Sh. Birman, M. Z. Solomjak,
{\it Spectral theory of selfadjoint operators in
Hilbert space}, D. Reidel Publishing Co., Dordrecht,
1987.

   \bibitem{CJJS2}
S. Chavan, Z.J. Jab\l o\'nski, I.B. Jung, J. Stochel,
Convergence of power sequences of B-operators with
applications to stability, {\em Proc. Amer. Math.
Soc.} {\bf 152} (2024), 2035-2050.

   \bibitem{CJJS23}
S. Chavan, Z. J. Jab\l o\'nski, I. B. Jung, J.
Stochel, Lifting $B$-subnormal operators, submitted
2023.

\bibitem{Co91} J. B. Conway, {\it The Theory of
Subnormal Operators}, Math. Surveys Monographs, {\bf
36}, Amer. Math. Soc. Providence, RI 1991.

\bibitem{Dow}
H. R. Dowson, {\it Spectral Theory of Linear
Operators}, Academic Press, London, 1978.

\bibitem{Eis10} T. Eisner, {\em Stability of operators and
operator semigroups}, Operator Theory: Advances and
Applications, 209, Birkh\"{a}user Verlag, Basel, 2010.

\bibitem{Fil}
P. A. Fillmore, {\it Notes on Operator Theory}, Van
Nostrand, New York, 1970.

\bibitem{Fuhr81}  P. A. Fuhrmann, {\em Linear systems and operators in Hilbert
space}, McGraw-Hill, New York, 1981.

\bibitem{Fur-Nak71} T. Furuta, R.  Nakamoto,
Certain numerical radius contraction operators, {\em
Proc. Amer. Math. Soc.} {\bf 29} (1971), 521-524.

\bibitem{Hal82} P. R. Halmos, {\em A Hilbert space problem
book}, Springer-Verlag, New York Inc. 1982.

\bibitem{Jab02} Z. J. Jab{\l}o\'{n}ski, Complete hyperexpansivity,
subnormality and inverted boundedness conditions, {\it
Integr. Equ. Oper. Theory} {\bf 44} (2002), 316-336.

\bibitem{Jab-St01} Z. J. Jab{\l}o\'{n}ski, J. Stochel,
Unbounded $2$-hyperexpansive operators, {\em Proc.
Edinburgh Math. Soc.} {\bf 44} (2001), 613-629.

   \bibitem{JJS23}
Z. J. Jab\l o\'nski, I. B. Jung, J. Stochel, Criteria
for algebraic operators to be unitary, {\em Kyungpook
Math. J.} {\bf 63} (2023), 1-10.

   \bibitem{Ju-St08} I. B. Jung, J. Stochel, Subnormal
operators whose adjoints have rich point spectrum,
{\em J. Funct. Anal.} {\bf 255} (2008), 1797-1816.

\bibitem{MDOT}
C. S. Kubrusly, {\em An Introduction to Models and
Decompositions in Operator Theory}, Birk\-h\"auser,
Boston, 1997.

\bibitem{Kub15} C. S. Kubrusly, {\em Essentials of
measure theory}, Springer, Cham, 2015.

\bibitem{Kub2}
C. S. Kubrusly, Singular-continuous unitaries and weak
dynamics, {\em Math. Proc. R. Ir. Acad.} {\bf 116A}
(2016), 45-56.

\bibitem{Kub89} C. S. Kubrusly, P. C.
M. Vieira; Weak asymptotic stability for discrete
linear distributed systems, Proceedings of the 5th
IFAC Symposium on Control of Distributed Parameter
Systems, Perpignan, pp. 69-73, June 1989 (Pergamon
Press, Oxford).

\bibitem{Ku-Vi08} C. S. Kubrusly; P. C.
M. Vieira, Convergence and decomposition for tensor
products of Hilbert space operators, {\em Oper.
Matrices} {\bf 2} (2008), 407-416.

\bibitem{Kub96} C. S. Kubrusly, P. C. M. Vieira,
D. O. Pinto, A decomposition for a class of
contractions, {\em Adv. Math. Sci. Appl.} {\bf 6}
(1996), 523-530.

\bibitem{Lyo}
R. Lyons, Seventy years of Rajchman measures,
Proceedings of the Conference in Honor of Jean-Pierre
Kahane (Orsay, 1993) {\em J. Fourier Anal. Appl.},
Special Issue (1995), 363-377.

\bibitem{Ml78}
W. Mlak, {\em Dilations of Hilbert space operators $($general
theory$)$}, Dissertationes Math. {\bf 153} (1978), 61 pp.

\bibitem{Nai43}
M. A. Naimark, On a Representation of Additive
Operator Set Functions,{\em C.R. Acad. Sci. URSS} {\bf
41} (1943), 359-361.

\bibitem{Ned20} L. Neder, \"{U}ber die
Fourierkoeffizienten der Funktionen von
beschr\"{a}nkter Schwankung, {\em Math. Z.} {\bf 6},
270-273 (1920).

\bibitem{Re-Si72} M. Reed, B. Simon,
{\em Methods of modern mathematical physics. I.
Functional analysis}, Academic Press, New York-London,
1972.

\bibitem{Rud}
W. Rudin, {\it Functional Analysis}, 2nd edn.
McGraw-Hill, New York, 1991.

\bibitem{shi74} A. L. Shields, Weighted shift operators
and analytic function theory, {\em Topics in operator
theory}, pp. 49-128. Math. Surveys, No. 13, Amer.
Math. Soc., Providence, R.I., 1974.

\bibitem{Sta66} J. G. Stampfli, Which weighted
shifts are subnormal, {\em Pacific J. Math.} {\bf 17}
(1966), 367-379.

\bibitem{Stam77}J. G. Stampfli, B. L. Wadhwa, On dominant operators,
{\em Monatsh. Math.} {\bf 84} (1977), 143-153.

\bibitem{St92} J. Stochel, Decomposition and
disintegration of positive definite kernels on convex
$*$-semigroups, {\em Ann. Polon. Math.} {\bf 56}
(1992), 243-294.

\bibitem{Sz.-N-Fo70} B. Sz.-Nagy, C. Foia\c{s},
{\em Harmonic analysis of operators on Hilbert space},
North-Holland Publishing Co., Amsterdam-London
American Elsevier Publishing Co., Inc., New York
Akad\'{e}miai Kiad\'{o}, Budapest, 1970.

\bibitem{Zab74} J. Zabczyk,
Remarks on the control of discrete-time distributed
parameter systems, {\em SIAM J. Control} {\bf 12}
(1974), 721-735.
\end{thebibliography}

\end{document}